\documentclass[smallcondensed, final]{svjour3}

\usepackage[english]{babel}
\usepackage[utf8]{inputenc} 
\usepackage{newtxtext}		
\usepackage[T1]{fontenc}	
\usepackage{amsfonts}		
\usepackage{mathrsfs}		
\usepackage{dsfont}			
\usepackage[dvipsnames]{xcolor}

\usepackage{pdfsync} 

\usepackage{amssymb, amsmath, amsthm}
\makeatletter
\def\TagsLeftOn{\tagsleft@true}\def\TagsLeftOff{\tagsleft@false}
\makeatother

\usepackage{verbatim} 
\usepackage{todonotes} 
\usepackage{url}
\usepackage{nameref}
\usepackage[numbers]{natbib} 
\usepackage{enumitem} 
\usepackage{algorithmic}
\usepackage{graphicx} 
\usepackage{subcaption} 
\usepackage[notcite,notref]{showkeys}

\definecolor{darkblue}{rgb}{0.0,0.0,0.6}
\usepackage[colorlinks=true,urlcolor=darkblue,citecolor=darkblue,linkcolor=darkblue]{hyperref}

\smartqed

\numberwithin{equation}{section}
\numberwithin{theorem}{section}
\numberwithin{figure}{section}
\numberwithin{table}{section}

\newtheorem{Lemma}[theorem]{Lemma}
\newtheorem{Proposition}[theorem]{Proposition}
\newtheorem{Corollary}[theorem]{Corollary}
\newtheorem{Definition}[theorem]{Definition}

\theoremstyle{remark}
\newtheorem{Remark}[theorem]{Remark}
\newtheorem{Example}[theorem]{Example}

\usepackage{tikz,pgfplots}
\usetikzlibrary{pgfplots.groupplots}
\usetikzlibrary{arrows.meta}
\usetikzlibrary{shapes,backgrounds}
\usetikzlibrary{patterns}

\makeatletter
\def\TagsLeftOn{\tagsleft@true}\def\TagsLeftOff{\tagsleft@false}
\def\namedlabel#1#2{\begingroup
    #2%
    \def\@currentlabel{#2}%
    \phantomsection\label{#1}\endgroup
}
\makeatother

\usepackage{my_latex_commands}
\theoremstyle{theorem}
\newtheorem*{algorithm}{Algorithm}

\newcommand{\ind}{{\tau,\delta}}
\newcommand{\zth}[1]{z_{#1}^\ind}
\newcommand{\tth}[1]{t_{#1}^\ind}
\newcommand{\Ih}[2]{\II(\tth{#1},\zth{#2})}

\newcommand{\V}{\mathbb{V}}
\newcommand{\ovbar}[1]{\overline{#1}}
\newcommand{\ubar}[1]{\underline{#1}}

\newcommand{\Ik}[2]{\II(t_{#1},z_{#2})}

\begin{document}


\title{Convergence Analysis of a Local Stationarity Scheme for Rate-Independent
Systems and Application to Damage\thanks{This research was supported by the German Research Foundation (DFG) under grant 
number~KN~1131/3-1 within the priority program \textsl{Non-smooth and Complementarity-based
Distributed Parameter Systems: Simulation and Hierarchical Optimization} (SPP~1962).}}
\titlerunning{Convergence Analysis of a Local Stationarity Scheme for RIS and Application to Damage}

\author{Michael Sievers}

\institute{
 M.~Sievers \at
 Technische Universit\"at Dortmund, Fakult\"at f\"ur Mathematik, 
 Vogelpothsweg 87, 
 44227 Dortmund, Germany \\
 Tel.: +49 (0)231 755 3231\\
 Fax: +49 (0)231 755 5416\\
 \email{michael.sievers@math.tu-dortmund.de}           
}

\date{Received: date / Accepted: date}

\maketitle

\begin{abstract}
This paper is concerned with an approximation scheme for rate-independent systems 
governed by a non-smooth dissipation and a possibly non-convex energy functional. 
The scheme is based on the local minimization scheme introduced in \cite{efenmielke06}, but relies on local stationarity of the underlying minimization problem. Under the assumption of Mosco-convergence for the dissipation functional, we show that accumulation points exist and are so-called parametrized solutions of the rate-independent system. In particular, this guarantees the existence of parametrized solutions for a rather general setting. Afterwards, we apply the scheme to a model for the evolution of damage.
\end{abstract}

\section{Introduction}
The effect of rate-independence occurs in various different areas of mechanics. This concerncs for example the field of elastoplasticity, damage and shape-memory, to only mention a few (see, e.g., \cite{krz13,francfort2006existence,mainik04phd,auricchio2008rate,mainik2009global}). One main characteristic of such systems is the fact that changes in the state are solely driven by an external force. What is more, as the name already suggests, the system is independent of the rate at which the loading is applied, that is to say, whenever $z$ is a solution to some external load $\ell$, then $z \circ \alpha$ is a solution to $\ell \circ \alpha$ for every monotone increasing function $\alpha$.\\ 
In this paper we consider rate-independent systems that can be described by the following differential inclusion
\begin{equation}
 0 \in \partial \RR(\dot z(t)) + D_z \II(t,z(t)) \qquad \ae \text{ in } [0,T] \, . \tag{RIS}
\label{subdiff-inclusion} 
\end{equation}
One may see this inclusion as a balance of forces, i.e. the dissipative force $\partial\RR$ and the potential force $-D_z\II(t,z)$ must annihilate each other. Implicitly hidden within this formulation is the fact that the potential force as well as the dissipative force result from a (possibly non-convex) energy functional $\II$ and a dissipation functional $\RR$. While we postpone the exact assumptions to Section~\ref{sec:assumptions}, let us mention at this point that the characteristic feature of the formulation in \eqref{subdiff-inclusion} is the positive $1$-homogeneity of the dissipation $\RR$. It is this property which induces that \eqref{subdiff-inclusion} is indeed rate-independent. However, the combination of non-convex energies and positive $1$-homogeneous dissipations allow the formation of abrupt changes in the state, even if the external forces evolve smoothly. Hence, suitable notions of solutions for \eqref{subdiff-inclusion} need to be able to handle temporal discontinuities. One such concept are the so-called \emph{parametrized solutions}, whose exact definition is given in Definition~\ref{def:paramsol}. Loosely speaking, such solutions are considered as curves in the extended phase space $[0,T] \times \ZZ$ and parametrized by arc-length. The jump path from one state to the other thus becomes an integral part of the solution itself. This idea was first applied in \cite{martins1994example,martins1995dissipative,bonfanti1996vanishing} for systems with dry friction and later on generalized in \cite{efenmielke06} and \cite{mielke2008modeling,mielkezelik} for finite and infinite dimensional problems, respectively. Particularly, in \cite{efenmielke06}, the authors introduced the following \emph{time-incremental local minimization scheme} for the approximation of parametrized solutions:
\begin{subequations}\label{eq:locminscheme}
\begin{align} 
 z_{k} &\in \argmin\{\II(t_{k-1},z) + \RR(z-z_{k-1}) \, : \, z \in \ZZ, \, \norm{z-z_{k-1}}_\VV \leq \tau\},  \label{eq:locmin}\\
 t_{k} &= \min\{t_{k-1}+\tau-\norm{z_{k}-z_{k-1}}_\VV,T\} . \label{eq:tupdate}
\end{align}
\end{subequations}
It was moreover shown that, for $\tau \searrow 0$, subsequences of discrete solutions generated by \eqref{eq:locminscheme} (weakly) converge to a parametrized solution. While the authors in \cite{efenmielke06} considered a finite dimensional setting, in \cite{Neg14,NegSca16} and particularly \cite{knees17} the results have been generalized to the infinite dimensional problems at least for semilinear energies. Furthermore, in \cite{fem_paramsol} the scheme was combined with a discretization in space. In this paper we extend these result to more general energy and dissipation functionals on the one side and, additionally, build the scheme upon stationary points rather then minimizer of \eqref{eq:locminscheme}. The actual scheme (\nameref{alg:locmin}) is presented in Section~\ref{sec:locmin}. Let us underline that the consideration of stationary points instead of (global) minimizers is of major importance from a numerical point of view, since optimization algorithms can in general only compute stationary points. What is more, we incorporate unbounded dissipations into our convergence analysis, which allows us to apply our scheme to a model for the evolution of damage. 
   

Now, let us shortly outline the paper. 
In Section~\ref{sec:assumptions}, we introduce our notation and state the assumptions on the energy 
and the dissipation functional. Moreover, we recall the precise notion of parametrized solutions. 
Section~\ref{sec:locmin} is then devoted to the presentation of the actual local stationarity scheme and its convergence analysis. Particularly, since we allow the dissipation to be approximated by some functional $\RR_\delta$, we provide suitably adapted a-priori estimates for the discrete solution which still meets a discrete version of an energy-equality. Building on that, we derive our main convergence result in Theorem~\ref{thm:main-conv}.
In Section~\ref{sec:num}, we then focus on a rate-independent damage model and describe the algorithmic realization of 
the discrete stationarity scheme based on a semismooth Newton-method. Finally, we present a numerical example and compare it with results from the literature. 

\section{Basic notations and standing assumptions}\label{sec:assumptions}

Let us start with some basic notation used throughout the paper. 
In the following, $C>0$ always stands for a generic constant. 
Moreover, given two normed linear spaces $X, Y$, we denote by $\dual{\cdot}{\cdot}_{X^*, X}$ 
the dual pairing and suppress the subscript, if there is no risk for ambiguity. 
By $\|\cdot\|_X$, we denote the norm in $X$ and $\LL(X,Y)$ is the space of linear and bounded 
operators from $X$ to $Y$. If $X$ is even a Hilbert space, we write $ J_X : X^* \to X $ for the Riesz isomorphism. Furthermore, $B_X(x,r)$ is the open ball in $X$ around $x\in X$ 
with radius $r>0$. Given a convex functional $f: X \to \R\cup\{\infty\}$, we denote the (convex) 
subdifferential of $f$ at $x$ by $\partial f(x)\subset X^*$ and its conjugate functional by 
$f^*: X^* \to \R\cup\{\infty\}$. 
Finally, $|\Omega|$ stands for the Lebesgue measure of a set $\Omega\subset \R^d$, $d\in \N$ and $\R^d_{\geq 0}$ describes the set of vectors in $\R^d$ whose components are greater or equal to $0$.

\subsection{Assumptions on the data}\label{sec:data}

Let us now introduce the assumptions on the quantities in \eqref{subdiff-inclusion}. We assume that the underlying spaces $\XX,\VV$ and $\ZZ $ are Banach spaces with $ \ZZ \embeds^{c,d} \VV \embeds \XX $. Moreover, $\ZZ$ and $\VV$ are required to be reflexive and separable.
 

\paragraph{\textbf{Energy}}\ \\
\noindent
The energy $\II(t,z)$ is supposed to fulfill:
\index{Assumption!(E1)-(E4): energy functional $\II$}
\begin{description}
\item[\namedlabel{ass:E1}{(E1)}] $\II \in C^1([0,T] \times \ZZ;\R)$.
\item[\namedlabel{ass:E2}{(E2)}] \emph{For all $t \in [0,T]$ the energy $\II(t,\cdot)$ is weakly lower semicontinuous and coercive on $\ZZ$ with $\II(t,z) \geq c_1 \norm{z}_\ZZ - c_0$ for some constants $ c_0,c_1 >0$.}
\item[\namedlabel{ass:E3}{(E3)}] \emph{There exists $\beta>0$ and $ \mu \in L^1(0,T)$ with $\mu \geq 0$ such that for all $t \in [0,T] $: \[ \abs{\partial_t \II(t,z)} \leq \mu(t) ( \II(t,z) + \beta) \quad \forall z \in \ZZ. \]}
\item[\namedlabel{ass:E4}{(E4)}] \emph{For all sequences $ t_k \to t $ and $ z_k \weakly z $ in $\ZZ$ it holds: \[ \partial_t \II(t_k,z_k) \to \partial_t\II(t,z). \]}
\end{description}
Note that the combination of \ref{ass:E1}--\ref{ass:E3} already yields that, for all sequences $t_k \to t$ and $ z_k \weakly z$ in $\ZZ$, it holds
\begin{equation}\label{eq:lower_semicont_combined}
\II(t,z) \leq \liminf_{k \to \infty} \II(t_k,z_k).
\end{equation}
Moreover, we assume that $\II$ satisfies the following G\r{a}rding-like inequality:
\begin{equation}\label{eq:garding-like}
\begin{gathered}
\forall z_1,z_2 \in \ZZ \textrm{ with } \norm{z_1}_\ZZ, \, \norm{z_2}_\ZZ \leq \rho \textrm{ there exists } c(\rho) \geq 0: \\
\dual{D_z\II(t,z_1)-D_z\II(t,z_2)}{z_1-z_2}_{\ZZ^*,\ZZ} \geq \alpha \norm{z_1-z_2}_\ZZ^2 - c(\rho) \norm{z_1-z_2}_\VV^2.
\end{gathered}
\end{equation}
With respect to the time component we require
\begin{equation}\label{eq:DzI-time-Lip}
\begin{gathered}
\forall z,v \in \ZZ \textrm{ with } \norm{z}_\ZZ \leq \rho \textrm{ there exists } C(\rho) \geq 0: \\ \dual{D_z\II(t,z)-D_z\II(s,z)}{v}_{\ZZ^*,\ZZ} \leq C(\rho) \abs{t-s} \norm{v}_\ZZ .
\end{gathered}
\end{equation}
Finally, we assume that $ D_z\II$ is (strong,weak)-weak-continuous, i.e.,
\begin{equation}\label{ass:I5}
\textrm{for all sequences } t_k \to t, \textrm{ and } z_k \weakly z \textrm{ in } \ZZ: \quad
D_z\II(t_k,z_k) \weakly D_z\II(t,z) \, \textrm{ in } \ZZ^* .
\end{equation} 

\noindent
The above assumptions combined with Gronwall's lemma allow us to obtain the following estimates that hold for all $t,s \in [0,T], \, z \in \ZZ$:
\begin{align}
 \II(t,z) + \beta & \leq (\II(s,z)+\beta) \exp\left( \int_s^t \mu(r) \dd r  \right) \label{eq:gronwall1} \\
 \text{ and } \quad \abs{\partial_t \II(t,z)} &\leq \mu(t) (\II(s,z)+\beta) \exp\left( \int_s^t \mu(r) \dd r \right) . 
 \label{eq:gronwall2}
\end{align}
An energy functional that satisfies all assumptions made here is given in Section~\ref{sec:num}.

\paragraph{\textbf{Dissipation}}\ \\
\noindent
Regarding the dissipation\index{dissipation!functional $\RR$} $\RR:\XX \to [0,\infty]$, we assume that 
\index{Assumption!(R1)-(R3): dissipation functional $\R$}
\begin{description}
\item[\namedlabel{ass:R1}{(R1)}] \emph{$\RR$ is proper, convex and lower semicontinuous,} 
\item[\namedlabel{ass:R2}{(R2)}] \emph{$\RR$ is positively $1$-homogeneous, i.e., $\RR(\lambda v) = \lambda \RR(v) \, \forall v \in \XX, \, \lambda > 0$, } 
\item[\namedlabel{ass:R3}{(R3)}] $\exists \, \kappa >0 : \, \kappa \, \norm{v}_\XX \leq \RR(v)$. 
\end{description}

\begin{Remark}
Note that we allow for an unbounded dissipation, which is essential for the application of our method to the damage model in Section~\ref{sec:app_num}.
\end{Remark}

Combining the convexity and the positive $1$-homogeneity of $\RR$, it is easy to verify the following triangle inequality
\begin{equation}\label{eq:R_triangle_ineq}
\RR(u-w) \leq \RR(u-v) + \RR(v-w) \quad \forall u,v,w \in \ZZ. 
\end{equation}
In fact we allow for an approximation of the "original" dissipation in our convergence analysis. In general, this corresponds to an approximation using e.g. finite elements. However, we will keep the setting general here, that is, we assume $ \RR_\delta : \XX \to [0,\infty]$ satisfies the same assumptions as $\RR$, i.e. \ref{ass:R1}-\ref{ass:R3}, and Mosco-converges to $\RR$ w.r.t. the space $\ZZ$ in the following sense:
\begin{subequations}\label{ass:Gamma-Conv-R}
\begin{align}
&\text{for all sequences } z_\delta \weakly z \, \text{in } \ZZ \, (\text{for } \delta \to 0): \quad \liminf_{\delta \to 0} \RR_\delta(z_\delta) \geq \RR(z) \\
&\forall z \in \ZZ \, \exists \text{ a sequence } z_\delta \to z \, \text{in } \ZZ \, (\text{for } \delta \to 0): \quad \limsup_{\delta \to 0} \RR_\delta(z_\delta) \leq \RR(z)
\end{align}
\end{subequations}

\begin{Remark}\label{rem:RR_only_in_ZZ}
Note that from now on, we consider $\RR$ and $\RR_\delta$, respectively, as mapping from $\ZZ$ into $[0,\infty]$. In fact, we will subsequently always evaluate $\RR$ at a point in $\ZZ$. Thus, the space $\XX$ is not used in the convergence analysis. Moreover, the choice $\RR_\delta = \RR$ 
(i.e., no additional approximation of $\RR$) is clearly possible and fulfills all these assumptions. 
Another example that complies with all the assumptions \ref{ass:R1}-\ref{ass:R3} and \eqref{ass:Gamma-Conv-R} but satisfies $\RR_\delta \not= \RR$ is given in Section~\ref{sec:FE} below.
\end{Remark}

\begin{Remark}\label{rem:subdiffRR_weakly_closed}
In view of the previous Remark~\ref{rem:RR_only_in_ZZ}, for any $z\in \ZZ$ the subdifferential $\partial\RR_\delta(z)$ is subsequently considered as a subset of $\ZZ^*$, i.e., by Lemma~\ref{lem:char.subdiff} we have $\xi \in \partial\RR_\delta(z)$ iff
\begin{align*}
\RR_\delta(w) &\geq \dual{\xi}{w}_{\ZZ^*,\ZZ} \quad \forall w \in \ZZ \\
\RR_\delta(z) &= \dual{\xi}{z}_{\ZZ^*,\ZZ}.
\end{align*}
In particular, in order to ease notation, we will refrain from using $\partial^\ZZ\RR_\delta(z)$ keeping in mind that $\partial\RR_\delta(z) \subset \ZZ^*$. Clearly, by the convexity and lower semicontinuity of $\RR_\delta$, the set $\partial\RR_\delta(0)$ is also weakly closed in $\ZZ^*$.
\end{Remark}

\paragraph{\textbf{Initial state}}\ \\
\noindent
The initial value $z_0$ is supposed to satisfy $z_0\in \ZZ$ and the local stability $-D_z\II(0,z_0) \in \partial\RR(0)$. Moreover, we assume that the approximations of the initial value as well satisfy $ -D_z\II(0,z_0^\delta) \in \partial \RR_\delta(0)$ for all $\delta > 0$ and that $ z_0^\delta $ is bounded in $\ZZ$ independent of $ \delta$. \vspace*{0.2cm}
\subsection{Definition of parametrized solutions}

We now turn to the actual definition of the so-called parametrized solutions. As indicated in the introduction, this concept takes care of possible jump paths and relies on an energy identity. 

\begin{Definition}\label{def:paramsol}
 Let an initial value $z_0 \in \ZZ$ be given. 
 We call a tuple $(\hat t, \hat z)$ \emph{parametrized solution} of \eqref{subdiff-inclusion}, 
 if there exists an artificial end time $S \geq T$ such that 
 the following conditions are satisfied:
 \begin{itemize}
  \item[(i)] Regularity:
  \begin{gather}
   \hat t \in W^{1, \infty}(0,S), \quad 
   \hat z \in W^{1,\infty}(0,S;\VV) \cap L^\infty(0,S;\ZZ)
  \end{gather}
  \item[(ii)] Initial and end time condition:
  \begin{equation}
  \hat t(0) = 0, \quad \hat z(0) = z_0, \quad \hat t(S) = T.
  \end{equation}
  \item[(iii)] Complementarity-like relations:
  \begin{subequations}\label{eq:compl}
  \begin{gather}
   \hat{t}^\prime(s) \geq 0 , \quad \hat{t}^\prime(s) + \norm{\hat{z}^\prime(s)}_\VV \leq 1, 
   \label{eq:degenerate} \\
   \hat{t}^\prime(s) \dist_{\VV^\ast}\{-D_z\II(\hat{t}(s),\hat{z}(s)),\partial \RR(0)\} = 0
   \quad \text{f.a.a.\ } s\in (0,S), \label{eq:slack}
  \end{gather}  
  \end{subequations}
  where $\dist_{\VV^*}\{\eta,\partial\RR(0)\} = \inf\{\norm{\eta- w}_{\VV^*}  : w \in \partial \RR(0)\}$, see also Lemma~\ref{lem:conv.ana}.
  \item[(iv)] Energy identity:
  \begin{equation}\label{eq:energyiddef}
  \begin{aligned}
   & \hspace*{-3ex}\II(\hat{t}(s),\hat{z}(s)) + \int_0^{s} \RR(\hat{z}^\prime(\sigma)) + \norm{\hat{z}^\prime(\sigma)}_\VV 
   \dist_{\VV^\ast}\{-D_z\II(\hat{t}(\sigma),\hat{z}(\sigma)),\partial \RR(0)\} \d \sigma \\
   &\qquad = \II(0,z_0)
   + \int_0^{s} \partial_t \II(\hat{t}(\sigma),\hat{z}(\sigma))\hat{t}^\prime(\sigma) \dd \sigma 
   \qquad \forall\, s\in [0,S].
  \end{aligned}
  \end{equation}
 \end{itemize}
 If, in addition to the second inequality in \eqref{eq:degenerate}, there is a constant $\gamma>0$
 such that $\hat{t}^\prime(s) + \norm{\hat{z}^\prime(s)}_\VV > \gamma$ f.a.a.\ $s\in (0,S)$, then the solution is called 
 \emph{non-degenerate parametrized solution}, otherwise we call it \emph{degenerate parametrized solution}.
\end{Definition}

We point out that it is always possible 
to rescale the artificial time in order to obtain a \emph{normalized parametrized solution}, 
where $\hat{t}^\prime(s) + \norm{\hat{z}^\prime(s)}_\VV = 1$ f.a.a.\ $s\in (0,S)$. The key idea here
is to cut out all intervals where $t^\prime(s)+\norm{z^\prime(s)}_\VV = 0$ and to scale the artificial time appropriately, see, e.g., \citep[Lem. A.4.3]{diss_sievers}. Moreover, let us mention that the regularity conditions in our definition of parametrized solutions are chosen in such a way, that all terms contained are well-defined. Depending on the actual setting, particularly the choice of $\RR$ and $\II$, there might exist slightly different requirements, see, e.g., \citep[Def. 4.2]{mrs12inf}.

\section{Local stationarity scheme}\label{sec:locmin}
The ultimate goal of this section is to prove that the subsequent algorithm, which is based on the local incremental minimization scheme \eqref{eq:locminscheme}, provides an approximation scheme for parametrized solutions. The difference compared to \eqref{eq:locminscheme} is that we search for stationary points of the constrained problem rather than global minima, see \eqref{eq:alg.zupdate}. Thus, for a given time-discretization parameter $T \geq \tau >0$, the algorithm reads as follows:

{
\begin{algorithm}[LISS]\label{alg:locmin}\ \\[-1ex]
\samepage
 \begin{algorithmic}[1]
  \STATE Let $z_0^\delta \in \ZZ$ be given with $ -D_z\II(0,z_0^\delta) \in \partial\RR_\delta(0)$. Set $t_0 = 0$, and $k = 1$.
  \WHILE{$t_k^\ind < T$}
   \STATE Compute a \emph{stationary point} $z_k^\ind$, i.e.,
   \begin{equation}
    0 \in \partial^{\ZZ}(\RR_\delta+I_\tau)(z_k^\ind-z_{k-1}^\ind) + D_z\II(t_{k-1}^\ind,z_k^\ind) \tag{alg$_1$} \label{eq:alg.zupdate} 
   \end{equation}
with the indicator function $I_\tau$ (see \eqref{eq:inditau}), which, additionally, satisfies 
   \begin{equation}\label{eq:alg.zupdate_lower}
   \II(t_{k-1}^\ind,z_k^\ind)+\RR_\delta(z_k^\ind-z_{k-1}^\ind) \leq \II(t_{k-1}^\ind,z_{k-1}^\ind). \tag{alg$_2$}
   \end{equation}
   \STATE Time update: 
   \begin{equation}
    t_k^\ind = t_{k-1}^\ind+\tau-\norm{z_k^\ind-z_{k-1}^\ind}_\V . \label{eq:alg.timeupdate} \tag{alg$_3$}
   \end{equation} 
   \STATE Set $k \to k+1$.
  \ENDWHILE
 \end{algorithmic}
 \nopagebreak
\end{algorithm}
}

Note that merely for technical reasons, we do not use the "$\min$" from \eqref{eq:tupdate} in the time-update. In addition, the notation $\partial^\ZZ$ is  used here only once more to highlight that the subdifferential is in fact calculated in terms of the space $\ZZ$, see also Remark~\ref{rem:subdiffRR_weakly_closed}. The proposed method is closely related to \eqref{eq:locminscheme}, since a local minimizer of 
\begin{equation}\label{eq:locmin_h}
\min\{\II(t_{k-1}^\ind,z) + \RR_\delta(z-z_{k-1}^\ind) \, : 
    \, z \in \ZZ, \, \norm{z- z_{k-1}^\ind}_\VV \leq \tau\}
\end{equation}
necessarily satisfies \eqref{eq:alg.zupdate}. Moreover, thanks to the assumptions on $\II$ and $\RR_\delta$, in particular weak lower semicontinuity, 
the existence of a global minimum of \eqref{eq:locmin_h} and therefore also the existence of a stationary point fulfilling \eqref{eq:alg.zupdate_lower} is guaranteed by the direct method in the calculus of variations.
The reason for investigating (\nameref{alg:locmin}) instead of \eqref{eq:locminscheme}, is the fact that a numerical algorithm for solving \eqref{eq:locmin} or rather \eqref{eq:locmin_h} naturally provides a stationary point $ z_k^\ind$ that satisfies $\II(t_{k-1}^\ind,z_k^\ind)+\RR_\delta(z_k^\ind-z_{k-1}^\ind) \leq \II(t_{k-1}^\ind,z_{k-1}^\ind)$ but, in case of a nonconvex energy, is not guaranteed to be a global optimum of \eqref{eq:locmin} and \eqref{eq:locmin_h}, respectively. Moreover, since the concept of parametrized solutions is based on a local stability condition, it is consistent to look for locally stable points, which are exactly the stationary points of \eqref{eq:locmin}. Despite its necessity for the convergence analysis, the inequality in \eqref{eq:alg.zupdate_lower} is also physically meaningful since it enforces the system to look for energetically preferable states, i.e., states with a lower energy cost. Concerning the exploration of this algorithm, particularly with a view to convergence, we proceed as follows: We start with characterizing properties of the stationary points. Afterwards, we turn to the essential a priori estimates that will allow a passage to the limit in the discrete version of the energy identity in \eqref{eq:energyiddef}, which is deduced in Lemma~\ref{prop:discrete-eneq}. The limit procedure itself is elaborated in the final Section~\ref{sec:main}.

\subsection{Approximate discrete parameterized solution}

The foundation for both, the a priori estimates and the discrete version of the energy identity, is given by the following Lemma~\ref{prop.optimalityprops}. It provides various properties of a stationary point $\zth{k}$ in \eqref{eq:alg.zupdate} and shows some similarities with the complementarity in \eqref{eq:compl}. Indeed, we will see that one can interpret the stationarity condition as a discrete version of \eqref{eq:compl}.

\begin{Lemma}[Discrete optimality System]\label{prop.optimalityprops} 
Let $k\geq 1$ and $\zth{k}$ be an arbitrary stationary point in the sense of \eqref{eq:alg.zupdate} 
with associated $\tth{k-1}$ given by \eqref{eq:alg.timeupdate}. 
Then the following properties are satisfied:
There exists a subgradient $\zeta_k^\ind \in \partial I_\tau(\zth{k}-\zth{k-1})$ such that
\begin{subequations}
\begin{gather}
 \norm{\zeta_k^\ind}_{\VV^\ast}(\norm{\zth{k}-\zth{k-1}}_\VV-\tau) = 0, \label{eq:opt.prop01} \\
 \tau \dist_{\VV^\ast}\{-D_z\Ih{k-1}{k},\partial \RR_\delta(0)\} 
 = \dual{\zeta_k^\ind}{\zth{k}-\zth{k-1}}_{\VV^\ast,\VV}, \label{eq:opt.prop02} \\
 \left.
 \begin{aligned}
  \RR_\delta(\zth{k}-\zth{k-1}) + \tau \dist_{\VV^\ast}\{-D_z\Ih{k-1}{k},\partial \RR_\delta(0)\} 
  \qquad\quad & \\
  = \dual{-D_z\Ih{k-1}{k}}{\zth{k}-\zth{k-1}}_{\ZZ^\ast,\ZZ} &
 \end{aligned} \; \right\}\label{eq:opt.prop03} \\
 \RR_\delta(v) \geq - \dual{\zeta_k^\ind
 + D_z\Ih{k-1}{k}}{v}_{\ZZ^\ast,\ZZ} \quad\forall v\, \in \ZZ. \label{eq:opt.prop04} 
\end{gather}
\end{subequations}
Herein, $\dist_{\VV^\ast}\{\, \cdot \, ,\partial \RR_\delta(0)\}$ denotes the extended distance as defined in Lemma~\ref{lem:conv.ana} 
\end{Lemma}
\begin{proof}
The proof is given, e.g., in \cite{fem_paramsol, diss_sievers} under a slightly different setting. For convenience of the reader, we thus repeat the main steps. To shorten the notation we also set $\RR_{\tau,\delta} = \RR_\delta + I_\tau$, cf. \eqref{eq:eq.A2}, and suppress the superscripts $ \tau, \delta$ for the iterates throughout the proof. 
\noindent
Thanks to a classical result of convex analysis, \eqref{eq:alg.zupdate} is equivalent to
\begin{equation}\label{eq:young}
\begin{aligned}
 \RR_{\tau,\delta}(z_{k}-z_{k-1}) &+ \RR_{\tau,\delta}^*(-D_z\II(t_{k-1}, z_{k}))\\ 
 &\qquad = \dual{-D_z\II(t_{k-1}, z_{k})}{z_{k}-z_{k-1}}_{\ZZ^\ast,\ZZ}\\
\end{aligned}
\end{equation}
Since $\norm{z_{k}-z_{k-1}}_\VV \leq \tau$ we have 
\begin{equation}\label{eq:RtauhRh}
 \RR_{\tau,\delta}(z_{k}-z_{k-1}) = \RR_\delta(z_{k}-z_{k-1}) . 
\end{equation}
Moreover, from Lemma~\ref{lem:conv.ana}, we infer
\[\RR_{\tau,\delta}^*(- D_z\II(t_{k-1},z_{k})) 
= \tau \dist_{\VV^\ast}\{-D_z\II(t_{k-1},z_{k}),\partial \RR_\delta(0)\}. \]
Inserting this together with \eqref{eq:RtauhRh} in \eqref{eq:young} gives \eqref{eq:opt.prop03}.\\
\noindent
To prove \eqref{eq:opt.prop01}, we consider \eqref{eq:alg.zupdate} once more. 
Since $ 0 \in \dom( \RR_\delta) \cap \dom(I_\tau)$ and $ I_\tau$ is continuous in $0$, the sum rule for convex
subdifferentials is applicable giving the existence of a $\zeta_k \in \partial I_\tau(z_{k}-z_{k-1})$, such that 
\begin{equation} 
 0 \in \partial \RR_\delta (z_{k}-z_{k-1}) + \zeta_k + D_z\II(t_{k-1}, z_{k}) 
 \label{dis-param-sol-inclusion} 
\end{equation}
and thereby
\begin{align*} 
 \RR_\delta(z_{k} - z_{k-1}) &+ \RR_\delta^*(-\zeta_k- D_z\II(t_{k-1}, z_{k})) \\
&\qquad = - \dual{\zeta_k+ D_z\II(t_{k-1}, z_{k})}{z_{k}-z_{k-1}}_{\ZZ^\ast,\ZZ} \\
&\qquad = - \dual{\zeta_k}{z_{k}-z_{k-1}}_{\VV^\ast,\VV} 
- \dual{D_z\II(t_{k-1}, z_{k})}{z_{k}-z_{k-1}}_{\ZZ^\ast,\ZZ}.
\end{align*}
A comparison with \eqref{eq:opt.prop03}, shows that
\begin{align} 
\begin{split}
 & \RR_\delta^* (-\zeta_k-D_z\II(t_{k-1}, z_{k})) \\
 & \qquad = \tau \dist_{\VV^\ast}\{- D_z\II(t_{k-1}, z_{k}),\partial \RR_\delta(0)\} 
 - \dual{\zeta_k}{z_{k}-z_{k-1}}_{\VV^\ast,\VV} . \label{eq:eq.aux00a}
\end{split}
\end{align}
Now, the fact that $\zeta_k \in \partial I_\tau(z_{k}-z_{k-1})$, and the characterization in Lemma~\ref{lem:subdiffItau} immediately yields \eqref{eq:opt.prop01}.
\noindent
Next, we verify \eqref{eq:opt.prop02}. For this purpose, we observe that by assumption $\RR_\delta$ is also convex and positively $1$-homogeneous so that Lemma~\ref{lem:char.subdiff} implies $ \partial\RR_\delta(z_{k}-z_{k-1}) 
 \subset \partial\RR_\delta(0)$. The characterization of the conjugate functional from Lemma~\ref{lem:char.subdiff} in combination with \eqref{dis-param-sol-inclusion} thus yields 
\begin{align}
 -\zeta_k - D_z\II(t_{k-1}, z_{k}) \in \partial \RR_\delta(z_{k}-z_{k-1}) 
 \subset \partial\RR_\delta(0) \qquad\quad & \label{eq:subdiffR0}\\
 \Longrightarrow \quad \RR_\delta^*(-\zeta_k- D_z\II(t_{k-1}, z_{k})) = 0. &
\end{align}
Inserting this into \eqref{eq:eq.aux00a} we arrive at \eqref{eq:opt.prop02}.
\noindent
Finally, \eqref{eq:opt.prop04} is an immediate consequence of \eqref{eq:subdiffR0}.

\end{proof}

\begin{Remark}
In fact, since \eqref{eq:alg.zupdate} is equivalent to the properties \eqref{eq:opt.prop01}--\eqref{eq:opt.prop04} it might be practical to exploit the characterization via \eqref{eq:opt.prop01}--\eqref{eq:opt.prop04} for the actual numerical realization of (\nameref{alg:locmin}) instead of \eqref{eq:alg.zupdate} in order to calculate a stationary point. Moreover, we will solely build upon this discrete optimality system (and the inequality \eqref{eq:alg.zupdate_lower}) for the convergence analysis.
\end{Remark}

Let us take a further look at \eqref{eq:opt.prop04}. Exploiting the properties of $ \zeta_k^\ind $ from \eqref{eq:subdiffItau_char} it is easy to see that $ \dual{\zeta_k^\ind}{\zth{k}-\zth{k-1}}_{\VV^\ast,\VV} = \tau \norm{\zeta_k^\ind}_{\VV^\ast}$. Inserting this into \eqref{eq:opt.prop04} we find that 
\begin{equation*} \dist_{\VV^\ast}\{-D_z\Ih{k-1}{k},\partial \RR_\delta(0)\} 
 = \norm{\zeta_k^\ind}_{\VV^\ast} .
 \end{equation*}
Combining this with \eqref{eq:opt.prop01} and the time-update $ t_k^\ind-t_{k-1}^\ind = \tau - \norm{z_k^\ind-z_{k-1}^\ind}_\V \geq 0$ from \eqref{eq:alg.timeupdate} we therefore obtain
\begin{equation}\label{eq:discr-complementarity}
\begin{gathered}
\frac{t_k^\ind-t_{k-1}^\ind}{\tau} \geq 0 \, , \quad \frac{t_k^\ind-t_{k-1}^\ind}{\tau} + \frac{\norm{z_k^\ind-z_{k-1}^\ind}_\VV}{\tau} = 1 \\
\left(\frac{t_k^\ind-t_{k-1}^\ind}{\tau}\right) \dist_{\VV^\ast}\{-D_z\Ih{k-1}{k},\partial \RR_\delta(0)\} = 0 
\end{gathered}
\end{equation}
which is a discrete version of the complementarity condition in the definition of parametrized solutions, cf. \eqref{eq:compl}.
\subsection{A-priori estimates}\label{sec:apriori}

Based on the previous Lemma~\ref{prop.optimalityprops}, we subsequently provide several a priori estimates that will allow a passage to the limit in the discrete energy identity in Section~\ref{sec:discrete-energy-eq} and \ref{sec:main}, respectively. Furthermore, we show that the discrete physical time $ t_k^\ind$ given the time update in \eqref{eq:alg.timeupdate} reaches the final time $T$ in a finite number of iterations, 
see Proposition~\ref{prop.finiteNumberSteps} below. We start with the following collection of results, whose proofs are basic so that we refer to \cite{fem_paramsol,knees17} here. Let us, nevertheless, remark that these a priori estimates are the only point where one needs to exploit that  $z_k^\ind$ is energetically preferred, that means \eqref{eq:alg.zupdate_lower} holds, and not only a stationary point satisfying \eqref{eq:opt.prop01}-\eqref{eq:opt.prop04}.

\begin{Lemma}[Boundedness for energy and dissipation] \label{lem.basicest1} 
For all $\delta,\tau>0$ and all $k \in \N$, it holds
\begin{gather}
\II(\tth{k},\zth{k}) + \sum_{i=1}^k \RR_\delta(\zth{i}-\zth{i-1}) \leq ( \beta +\II(0,z_0^\ind)) \exp\left(\int_0^T \mu(s) \dd s \right), \label{eq:basic.02} \\
 \sup_{\delta,\tau>0,\, k\in\N} \norm{\zth{k}}_\ZZ < \infty . 
 \label{eq:basic.03}
\end{gather}
where $\beta$ and $\mu$ are the components from Section~\ref{sec:assumptions}.
\end{Lemma}
\begin{proof}
The proof mainly relies on the estimates in \eqref{eq:gronwall1} - \eqref{eq:gronwall2} and the coercivity of the energy from assumption \ref{ass:E2}, cf. \cite{fem_paramsol,knees17}.
\end{proof}

\begin{Remark}\label{rem:zk_in_ball}
As a consequence of Lemma~\ref{lem.basicest1} and the boundedness of $z_0^\delta$ by assumption we have that $z_k^\ind \in B_\ZZ(0,R) $ for some $R >0$ independent of $\tau$ and $\delta$.
\end{Remark}

The estimate \eqref{eq:basic.03} will, on the one hand, provide us with a uniform $L^\infty$-bound for the linear interpolants and, on the other hand, allows us to obtain a bound for the derivative $D_z\II$. In preparation for that, we derive the following:

\begin{Lemma}\label{lem:garding-combined}
For every $\rho>0$, there exists $C_1(\rho),C_2(\rho)>0$, such that 
\begin{multline}\label{eq:garding-combined}
\dual{D_z\II(t,v)-D_z\II(s,w)}{v-w}_{\ZZ^*, \ZZ} \\
\geq \frac{\alpha}{2} \norm{v-w}_\ZZ^2 - C_1(\rho) \, \norm{v-w}_\VV \, \RR_\delta(v-w) - C_2(\rho) \, (t-s)^2
\end{multline}
for all $v,w \in B_\ZZ(0,\rho)$ and $ t,s \in [0,T]$.
\end{Lemma}

\begin{proof}
According to Ehrling's lemma, for every $\varepsilon >0$, there exists a constant $C_\varepsilon$ such that 
\begin{equation} 
 \norm{z}_\VV \leq \varepsilon \norm{z}_\ZZ + C_{\varepsilon} \norm{z}_\XX \quad \forall\, z\in \ZZ.
 \label{eq:ehrling_ineq} 
\end{equation}
Combining the Garding-like inequality from \eqref{eq:garding-like} with \eqref{eq:DzI-time-Lip} we find
\begin{multline}\label{eq:aux-differenceDzI}
\dual{D_z\II(t,v)-D_z\II(s,w)}{v-w}_{\ZZ^*,\ZZ} \\ \geq \alpha \norm{v-w}_\ZZ^2 - c(\rho) \norm{v-w}_\VV^2 - C(\rho) \abs{t-s} \norm{v-w}_\ZZ.
\end{multline}
To proceed, we consider each of the two last terms separately. For the first one, we exploit \eqref{eq:ehrling_ineq} for $\varepsilon = \tfrac{\alpha}{4 c_{\VV} c(\rho)} $ (recall that $ c_\VV$ denotes the embedding constant of $\ZZ \embeds \VV$) to obtain
\begin{equation}\label{eq:diff-zk-VV}
\begin{aligned}
& c(\rho) \norm{v-w}_\VV^2 \\
&\qquad \leq ( \frac{\alpha}{4 c_\VV} \norm{v-w}_\ZZ + C(\alpha,\rho,c_\VV) \norm{v-w}_\XX ) \, \norm{v-w}_\VV \\
 &\qquad \leq \frac{\alpha}{4} \norm{v-w}_\ZZ^2 + C(\alpha,\rho,c_\VV,\kappa) \, \RR_\delta(v-w) \norm{v-w}_\VV 
\end{aligned}
\end{equation}
where we used the lower bound for $\RR_\delta$ from \ref{ass:R3} and the embedding $\ZZ \embeds \VV$ in the last line. Next, we turn to the last term in \eqref{eq:aux-differenceDzI}. For this, we take advantage of Young's inequality which gives
\begin{equation}\label{eq:diff-tk-zk-ZZ}
C(\rho) \, \abs{t-s} \, \norm{v-w}_\ZZ \leq \frac{\alpha}{4} \norm{v-w}_\ZZ^2 + C(\alpha,\rho) \, (t-s)^2 .
\end{equation}
Inserting \eqref{eq:diff-zk-VV} and \eqref{eq:diff-tk-zk-ZZ} in \eqref{eq:aux-differenceDzI} we eventually arrive at \eqref{eq:garding-combined}.

\end{proof}

Clearly, from the uniform boundedness of the iterates and $ \norm{z_{k+1}-z_k}_\VV \leq \tau$ we can conclude the following result.

\begin{Corollary}\label{lem:garding-combined-iter}
For all iterates $ z_k \in \ZZ$ there exists constants $C_1,C_2>0$ such that
\begin{multline}\label{eq:garding-combined-iter}
\dual{D_z\II(t_{k},z_{k+1})-D_z\II(t_{k-1},z_k)}{z_{k+1}-z_k}_{\ZZ^*, \ZZ} \\
\geq \frac{\alpha}{2} \norm{z_{k+1}-z_k}_\ZZ^2 - C_1 \, \tau \, \RR_\delta(z_{k+1}-z_k) - C_2 \, \tau \, (t_k-t_{k-1}) .
\end{multline}
\end{Corollary}
\begin{proof}
By Remark~\ref{rem:zk_in_ball} all iterates are bounded by some constant, i.e., $z_k \in B_\ZZ(0,R)$ for some $ R >0 $ for every $k \in \N$. Hence, combining \eqref{eq:garding-combined} with $ \norm{z_{k+1}-z_k}_\VV \leq \tau$ and $t_k-t_{k-1} \leq \tau$, we immediately have \eqref{eq:garding-combined-iter}.
\end{proof}

One major issue in the convergence analysis for parametrized 
solutions concerns the boundedness of the artificial time, even in the continuous setting, see, e.g., the discussion in \citep[p. 218]{mielkeroubi}. For the discrete counterpart, the artificial time reads $ s_n = \sum_{k=1}^n t_k - t_{k-1} + \norm{z_k-z_{k-1}}_\VV $. In order to bound this term, we need to estimate $ \sum_{k=1}^n \norm{z_k-z_{k-1}}_\VV $, which is purpose of the next proposition. Moreover, we will show that the physical end time $T$ is reached after a finite number of iterations, which guarantees that the algorithm finishes in a finite number of steps.
\begin{Proposition}[Bound on artificial time] \label{prop.finiteNumberSteps}
For every parameter $\delta,\tau>0$ there exists an index $N(\tau,\delta) \in \N$ such that $\tth{N(\tau,\delta)} \geq T$.
Moreover, there are constants $C_1,C_2,C_3>0$ independent of $\tau,\delta$ such that, for all $\delta,\tau>0$, it holds
\begin{align}
 \sum_{i=1}^{N(\tau,\delta)} \norm{\zth{i}-\zth{i-1}}_\VV &\leq C_1,  \label{eq:eq.boundSteps}\\
 \sum_{i=1}^{N(\tau,\delta)} \norm{\zth{i}-\zth{i-1}}_\ZZ^2 &\leq C_2 \, \tau, 
   \label{eq:eq.est_z_H1} \\
  \text{and} \quad \dist_{\VV^\ast}\{-D_z\Ik{k-1}{k}, \partial \RR_\delta(0)\} &\leq C_3 
   \quad \forall\, k = 1, ..., N(\tau,\delta). \label{eq:eq.est_dist} 
\end{align}
\end{Proposition}\begin{proof}
The arguments are similar to \cite{knees17,fem_paramsol}. However, since there are some significant differences, particularly the estimate \eqref{eq:eq.est_z_H1}, we present the arguments in detail. Let $k \in \N$ be arbitrary. For convenience, we again suppress the superscript $\tau, \delta$ 
throughout the proof, except for $z_0^\ind$ in order to avoid confusion with the initial data. 
We start by testing \eqref{eq:opt.prop04} with $v = z_{k+1} - z_{k}$ to obtain
\begin{equation}\label{eq:eq1.aux006}
\begin{aligned}
 \RR_\delta(z_{k+1} - z_{k}) \geq - \dual{\zeta_k}{z_{k+1} - z_{k}}_{\VV^\ast,\VV} 
 - \dual{D_z\Ik{k-1}{k}}{z_{k+1} - z_{k}}_{\ZZ^\ast,\ZZ}. 
\end{aligned}
\end{equation}
Inserting \eqref{eq:opt.prop02} into \eqref{eq:opt.prop03} 
and rewriting this identity for the index $k+1$ (instead of $k$) gives $ \RR_\delta(z_{k+1}-z_{k}) + \dual{\zeta_{k+1}}{z_{k+1}-z_{k}}_{\VV^\ast,\VV} 
= \dual{-D_z\Ik{k}{k+1}}{z_{k+1}-z_{k}}_{\ZZ^\ast,\ZZ} $. Subtracting this from \eqref{eq:eq1.aux006} implies
\begin{multline*}
 0 \geq \dual{\zeta_{k+1}}{z_{k+1}-z_{k}}_{\VV^\ast,\VV} 
 - \dual{\zeta_k}{z_{k+1}-z_{k}}_{\VV^\ast,\VV} \\
 + \dual{D_z\Ik{k}{k+1} - D_z\Ik{k-1}{k}}{z_{k+1} - z_{k}}_{\ZZ^\ast,\ZZ} \, .
\end{multline*}
Thanks to the constraint $\norm{z_{k+1}-z_{k}}_\VV \leq \tau$ and \eqref{eq:subdiffItau_char}, that is $ \norm{\zeta_{k}}_{\VV^*} \, \tau = \dual{\zeta_k}{z_k-z_{k-1}}_{\VV^*,\VV}$, we have 
\[ \dual{\zeta_{k}}{z_{k+1}-z_{k}}_{\VV^\ast,\VV} \leq \norm{\zeta_{k}}_{\VV^*} \norm{z_{k+1}-z_{k}}_{\VV} \leq \norm{\zeta_{k}}_{\VV^*} \, \tau = \dual{\zeta_k}{z_k-z_{k-1}}_{\VV^*,\VV}  . \]
Consequently, it holds
\begin{multline*}
 0 \geq \dual{\zeta_{k+1}}{z_{k+1}-z_{k}}_{\VV^\ast,\VV} 
 - \dual{\zeta_k}{z_{k}-z_{k-1}}_{\VV^\ast,\VV} \\
 + \dual{D_z\Ik{k}{k+1} - D_z\Ik{k-1}{k}}{z_{k+1} - z_{k}}_{\ZZ^\ast,\ZZ} \, .
\end{multline*}
Now, inserting the estimate from Corollary~\ref{lem:garding-combined-iter} gives 
\begin{align*}
0 &\geq \dual{\zeta_{k+1}}{z_{k+1}-z_{k}}_{\VV^\ast,\VV} 
 - \dual{\zeta_k}{z_{k}-z_{k-1}}_{\VV^\ast,\VV} \\ & \qquad + \frac{\alpha}{2} \norm{z_{k+1}-z_k}_\ZZ^2 - C_1 \, \tau \, \RR_\delta(z_{k+1}-z_k) - C_2 \, \tau \, (t_k-t_{k-1}) 
\end{align*}
Rearranging terms and summing up the resulting estimate with respect to $k$, we arrive at
\begin{multline}\label{eq:eq.aux008}
 \dual{\zeta_{k+1}}{z_{k+1}-z_{k}}_{\VV^\ast,\VV}  + c \sum_{i=1}^k\norm{z_{i+1}-z_{i}}_{\ZZ}^2 \\
\leq \dual{\zeta_{1}}{z_{1}-z_{0}}_{\VV^\ast,\VV}  + C \tau \Big( t_{k} + \sum_{i=1}^k \RR_\delta(z_{i+1}-z_{i}) \Big) .
\end{multline}
Thanks to \eqref{eq:basic.02} it now suffices to estimate $\dual{\zeta_{1}}{z_{1}-z_{0}}_{\VV^\ast,\VV}$ to proof \eqref{eq:eq.est_z_H1}, which is shown next. To this end, we again insert \eqref{eq:opt.prop02} into \eqref{eq:opt.prop03} to obtain for $k=1$:
\[ \RR_\delta(z_1-z_0) + \dual{\zeta_{1}}{z_{1}-z_{0}}_{\VV^\ast,\VV} = \dual{-D_z\II(0,z_{1})}{z_{1}-z_0^\ind}_{\ZZ^\ast,\ZZ} \, .\]
Adding a zero, and rearranging terms yields
\begin{multline} \label{eq:alternative_lambda1}
\dual{-D_z\II(0,z_0^\ind)}{z_{1}-z_0^\ind}_{\ZZ^\ast,\ZZ} \geq \dual{D_z\II(0,z_{1})-D_z\II(0,z_0^\ind)}{z_{1}-z_0^\ind}_{\ZZ^\ast,\ZZ} \\ 
+ \RR_\delta(z_1-z_0^\ind) +  \dual{\zeta_{1}}{z_{1}-z_{0}}_{\VV^\ast,\VV} .
\end{multline}
By assumption, we have $ -D_z\II(0,z_0^\ind) \in \partial\RR_\delta(0) $ which gives $\dual{-D_z\II(0,z_0^\ind)}{z_1-z_0^\ind}_{\ZZ^*,\ZZ} \leq \RR_\delta(z_1-z_0^\ind) $ by the characterization in Lemma~\ref{lem:conv.ana}, so that \eqref{eq:alternative_lambda1} implies
\begin{equation} \label{eq:alternative_lambda1a}
\dual{D_z\II(0,z_{1})-D_z\II(0,z_0^\ind)}{z_{1}-z_0^\ind}_{\ZZ^\ast,\ZZ} +  \dual{\zeta_{1}}{z_{1}-z_{0}}_{\VV^\ast,\VV}
 \leq 0 .
\end{equation}
For the first term on the left-hand side we take advantage of the assumption in \eqref{eq:garding-like} and the boundedness of the iterates from Lemma~\ref{lem.basicest1} which results in
\begin{equation} \label{eq:alternative_lambda2}
 \alpha \norm{z_{1}-z_0^\ind}^2_\ZZ - C \norm{z_1-z_0^\ind}_\VV^2 +  \dual{\zeta_{1}}{z_{1}-z_{0}}_{\VV^\ast,\VV}
 \leq 0.
\end{equation}
Hence, using again the constraint $\norm{z_{k+1}-z_{k}}_\VV \leq \tau$ we obtain
\begin{equation}\label{eq:kgleich1}
  \dual{\zeta_{1}}{z_{1}-z_{0}}_{\VV^\ast,\VV} + c\|z_1 - z_0^\ind\|_\ZZ^2
 \leq C\,  \tau^2 .
\end{equation}
By adding \eqref{eq:kgleich1} to \eqref{eq:eq.aux008} and applying \eqref{eq:basic.02}, we arrive at
\begin{equation}\label{eq:sum_bound}
\begin{aligned}
&\dual{\zeta_{k+1}}{z_{k+1}-z_{k}}_{\VV^\ast,\VV} + c \sum_{i=0}^k\norm{z_{i+1}-z_{i}}_{\ZZ}^2 \\
 &\qquad \qquad \leq C \, \tau \, \Big( t_{k} + \sum_{i=0}^k \RR_\delta(z_{i+1}-z_{i}) 
 + \tau \Big) \\
 &\qquad \qquad \leq C \, \tau \Big(T + (\II(0,z_0^\ind) + \beta)\exp\left(\int_0^T \mu(s) \dd s \right) \Big),
\end{aligned}
\end{equation}
where we used that $t_k \leq T + \tau \leq 2 T $ by the time update in \eqref{eq:alg.timeupdate} and $ \tau \leq T$ for the last estimate. Clearly, by the uniform boundedness of $z_0^\ind$ by assumption and the continuity of $\II(0,\cdot)$, the term $\II(0,z_0^\ind)$ is also bounded
independent of $\tau$ and $\delta$, which yields that $T + (\II(0,z_0^\ind) + \beta)\exp\left(\int_0^T \mu(s) \dd s \right) \leq C $. This in turn implies
\begin{equation}\label{eq:boundprelim}
  \dual{\zeta_{k+1}}{z_{k+1}-z_{k}}_{\VV^\ast,\VV} + c \sum_{i=0}^k\norm{z_{i+1}-z_{i}}_{\ZZ}^2 \leq C \, \tau,
\end{equation}
which already gives \eqref{eq:eq.est_z_H1} for $k\geq 0$ since $ \dual{\zeta_{k+1}}{z_{k+1}-z_{k}}_{\VV^\ast,\VV} \geq 0$ by \eqref{eq:opt.prop02}. Moreover, \eqref{eq:eq.est_dist} is an easy consequence of the characterization in \eqref{eq:opt.prop02}. Note that the constant $C$ is independent of $\tau$, $\delta$, and $k$. 
%
Now, let us turn towards \eqref{eq:eq.boundSteps}. From the identity \eqref{eq:opt.prop02} we infer that $ \dual{\zeta_{k+1}}{z_{k+1}-z_{k}}_{\VV^\ast,\VV} \geq 0$. Moreover, thanks to the time-update it holds $t_{k+1}-t_k + \norm{z_{k+1}-z_{k}}_\VV = \tau$ so that
\begin{align}
\sum_{i=0}^k \norm{z_{i+1}-z_{i}}_{\VV} &= \frac{1}{\tau} \, \sum_{i=0}^k \norm{z_{i+1}-z_{i}}_{\VV} \, \big(t_{i+1}-t_i + \norm{z_{i+1}-z_{i}}_\VV \big) \notag \\
&= \sum_{i=0}^k (t_{i+1}-t_i)\frac{\norm{z_{i+1}-z_{i}}_{\VV}}{\tau} + \frac{1}{\tau} \, \sum_{i=0}^k \norm{z_{i+1}-z_{i}}_{\VV}^2 \notag \\
&\leq \sum_{i=0}^k (t_{i+1}-t_i) + \frac{1}{\tau} \, \sum_{i=0}^k \norm{z_{i+1}-z_{i}}_{\VV}^2 \notag \\
&\leq C \label{eq:est.Vnorm_k}
\end{align}
where we used \eqref{eq:boundprelim} together with the embedding $ \ZZ \embeds \VV$ as well as the fact that $ t_{k+1} \leq T + \tau \leq 2T$ for the last estimate. This verifies \eqref{eq:eq.boundSteps}. 
Finally, we show that the final time $T$ is reached after a finite number of steps. For this, we observe that by the embedding $\ZZ \embeds \VV$ estimate \eqref{eq:boundprelim} implies that $\sum_{k=1}^\infty \norm{z_k-z_{k-1}}_\VV$ is convergent, thus bounded. Summing up \eqref{eq:alg.timeupdate} from $k=1$ to $n$ and exploiting \eqref{eq:eq.boundSteps} we therefore obtain
\begin{equation*}
t_n = t_0 + n \tau - \sum_{k=1}^n \norm{z_k-z_{k-1}}_\VV \geq t_0 + n \tau - C \to \infty \quad \text{for } n \to \infty.
\end{equation*}
Hence, there must exist a finite index $N(\tau,\delta)$, possibly depending on $\tau$ and $\delta$, so that $t_{N(\tau,\delta)} \geq T$. Lastly, since \eqref{eq:boundprelim} and \eqref{eq:est.Vnorm_k} hold for every $k$, we obtain \eqref{eq:eq.boundSteps} and \eqref{eq:eq.est_z_H1}, respectively.
\end{proof}

In what follows we will abbreviate the index $N(\tau,\delta)$ simply by $N$ having in mind that the number $N$ of time steps always depends on $\tau$ and $\delta$.

\subsection{Discrete energy-equality}\label{sec:discrete-energy-eq}

In the following section, we aim at deriving a discrete analogon to the energy identity \eqref{eq:energyiddef}.
To this end, we introduce the piecewise affine as well as the left- and right-continuous 
piecewise constant interpolants associated with the iterates $z_k^\ind$. 
As indicated in the introduction, potential discontinuities 
of the parametrized solution are resolved by introducing an artificial time. The physical time is accordingly interpreted as a function of the very same and jumps are characterized by the plateaus of this function. This is also reflected by the time-incremental stationarity scheme (\nameref{alg:locmin}), where, loosely speaking, the artificial time is divided into equidistant subintervals with step size $\tau$ and the approximation of the parametrized solution is implicitly defined through the optimization in (\nameref{alg:locmin}). To be more precise, we set $ s_k^\ind := k\tau$, so that
\begin{equation}\label{eq:s_N-def}
\begin{aligned}
 s_N^\ind = N \tau &= \sum_{i=1}^N (t_i^\ind - t_{i-1}^\ind + \norm{z_i^\ind-z_{i-1}^\ind}_\VV ) \\[-2ex]
 &= t_N^\ind + \sum_{i=1}^N \norm{z_i^\ind-z_{i-1}^\ind}_\VV \leq T + \tau +  \sum_{i=1}^N \norm{z_i^\ind-z_{i-1}^\ind}_\VV \leq C_S
\end{aligned}
\end{equation}
by Proposition~\ref{prop.finiteNumberSteps} with a constant $C_S>0$ which is neither depending on $\tau$ nor $\delta$ so that the artificial time interval is indeed bounded. Hence, we can proceed with the construction of the interpolants. For $s \in [s_{k-1}^\ind,s_k^\ind) \subset [0,s_N^\ind)$, the continuous and piecewise affine interpolants are defined through
\begin{equation}
\begin{aligned}
\hat{z}_\ind(s) &:= \zth{k-1} + \frac{(s-s_{k-1}^\ind)}{\tau} ( \zth{k} - \zth{k-1}), \\
\hat{t}_\ind(s) &:= \tth{k-1} + \frac{(s-s_{k-1}^\ind)}{\tau} ( \tth{k} - \tth{k-1}), 
\label{eq:affine-interpolants} \\
\end{aligned}
\end{equation}
while the piecewise constant interpolants are given by
\begin{equation}\label{eq:constinter}
 \ovbar{z}_\ind(s) := \zth{k}, \qquad \ovbar{t}_\ind(s) := \tth{k}, \qquad 
 \ubar{z}_\ind(s) := \zth{k-1}, \qquad \ubar{t}_\ind(s) := \tth{k-1}.
\end{equation}
Moreover, we define the artificial end time $ S_\ind $ as that point where $ \hat{t}$ reaches the end time $T$, i.e., it holds (see also Figure~\ref{fig:affine_interpol_t})
\begin{equation}\label{eq:Sdef}
\hat{t}_\ind(S_\ind)=T, \quad s_{N-1}^\ind < S_\ind \leq s_N^\ind \quad \text{and} \quad S_\ind \leq C_S,
\end{equation}
whereby the boundedness follows directly from \eqref{eq:s_N-def}. 
Since the artificial end time $S_\ind$ depends on the chosen discretization level, 
we extend all interpolants constantly onto $[0,\tilde{S}]$ with $\tilde{S} := \sup_{\tau,h} S_\ind$ where this is necessary, i.e., where $ s_N^\ind < \tilde{S}$. Hence, we let
\begin{equation}\label{eq:affine-interpolants-extended}
\left.
\begin{alignedat}{4}
\ovbar{z}_\ind(s) &= \ubar{z}_\ind(s) &= \hat{z}_\ind(s) &:= \zth{N} \\
\text{and} \quad  \ovbar{t}_\ind(s) &= \ubar{t}_\ind(s) &= \hat{t}_\ind(s) &:= T
\end{alignedat}
\quad \right\} \quad \forall\, s \in [s_N^\ind,\tilde{S}] \,.
\end{equation}
Observe that still $\tilde{S} \leq C_S$ by \eqref{eq:Sdef}. Moreover, due to the time update in \eqref{eq:alg.timeupdate}, we clearly have that $(\hat{t}_\ind,\hat{z}_\ind) \in W^{1,\infty}(0,\tilde{S};\R) \times W^{1,\infty}(0,\tilde{S};\VV)$, but we even obtain the following pointwise properties. 
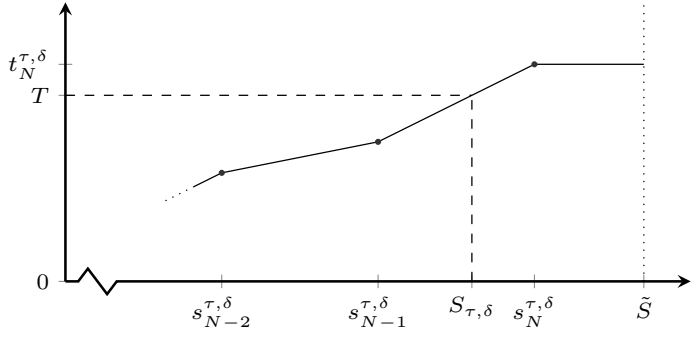
\begin{figure}
\centering
\begin{tikzpicture}[scale = 1.2]
     \begin{axis}[xmin=0,xmax=20,ymin=0,ymax=9,
        axis lines = left, axis line style = thick,  unit vector ratio=1 1 1, 
        axis x discontinuity=crunch,
        tick label style={font=\scriptsize},
        	xtick={5,10,13,15,18.5},
        xticklabels = {$s_{N-2}^\ind$, $s_{N-1}^\ind$, $S_\ind$, $s_N^\ind$, $\tilde{S}$},
        ytick={0,6,7},
        yticklabels = {$0$, $T$,$t_N^\ind$}]
        \addplot[mark=none, dashed] coordinates {(13, 0) (13, 6)}; 
        \addplot[domain=0:13, dashed] {6};
        \addplot[mark=none, dotted] coordinates {(18.5,0) (18.5,10)};
		\addplot[mark=none, dotted] coordinates {(3.2, 2.6) (4.1, 3.05)};
        \addplot[mark=none] coordinates {(4.1, 3.05) (5, 3.5)};
        \addplot[mark=none] coordinates {(5, 3.5) (10, 4.5)};
        \addplot[mark=none] coordinates {(10, 4.5) (15, 7)};
        \addplot[mark=none] coordinates {(15, 7) (18.5, 7)};

        \node at (axis cs:5,3.5) [circle, scale=0.2, draw=black!80,fill=black!80] {};
        \node at (axis cs:10,4.5) [circle, scale=0.2, draw=black!80,fill=black!80] {};
        \node at (axis cs:15,7) [circle, scale=0.2, draw=black!80,fill=black!80] {};
      \end{axis}
\end{tikzpicture}
\caption{Qualitative illustration of the affine interpolant $\hat{t}$, the choice of the artificial end time $S_\ind$ via the equality $\hat{t}(S_\ind)=T$ and the upper bound $\tilde{S}$.}
\label{fig:affine_interpol_t}
\end{figure}

\begin{Lemma}[Properties of affine interpolants] \label{lem.prop-affine-interpol}
For almost all $ s \in [0,S_\ind]$, the affine interpolants from \eqref{eq:affine-interpolants} fulfill
\begin{gather}
\hat{t}^\prime_\ind(s) \geq 0 , \qquad \hat{t}_\ind^\prime(s) + \norm{\hat{z}_\ind^\prime(s)}_\VV = 1, \label{eq:affine.aux1}\\
\hat{t}_\ind^\prime(s) \dist_{\VV^\ast}\{-D_z\II(\ubar{t}_\ind(s),\ovbar{z}_\ind(s)), \partial \RR_\delta(0)\} = 0.
 \label{eq:affine.aux2}
\end{gather}
\end{Lemma}

\begin{proof}
The statements are a direct consequence of the properties in \eqref{eq:discr-complementarity}.
\end{proof}

Once more, we note the similarity between the continuous case in \eqref{eq:degenerate} and \eqref{eq:slack} and its discrete version in Lemma~\ref{lem.prop-affine-interpol}. In the subsequent, last preparatory lemma, we collect the main a priori bounds of our interpolants, which will be essential to pass to the limit in the discrete energy identity, which is elaborated afterwards.

\begin{Lemma}\label{lem:unibound}
There exists $C>0$, independent of $\tau$ and $\delta$, so that 
 \begin{equation*}
\|\hat t_\ind\|_{W^{1,\infty}(0,\tilde S)}, \|\hat z_\ind\|_{W^{1,\infty}(0,\tilde S; \VV)}, \|\hat z_\ind\|_{L^{\infty}(0,\tilde S; \ZZ)}, \|\hat z_\ind\|_{H^{1}(0,\tilde S; \ZZ)}  \leq C.
 \end{equation*}
\end{Lemma}
\begin{proof}
While the first three bounds are an immediate consequence of the results in Lemma~\ref{lem.prop-affine-interpol} and Lemma~\ref{lem.basicest1}, the last one requires some slighlty more explanation. Due to the bound in $L^\infty(0,\tilde{S};\ZZ)$, it suffices to estimate the $L^2(0,\tilde{S};\ZZ)$-norm of the time-derivative $\hat{z}_\ind^\prime$. Hence, inserting the definition of $\hat{z}$ from \eqref{eq:affine-interpolants-extended} and keeping in mind that $S_\ind \leq s_N^\ind$, we have
\begin{multline*}
\norm{\hat{z}_\ind^\prime}_{L^2(0,\tilde{S};\ZZ)}^2 
= \int_0^{S_\ind} \norm{\hat{z}_\ind^\prime(r)}_\ZZ^2 \dd r \\ 
\leq \sum_{k=1}^N \int_{s_{k-1}^\ind}^{s_k^\ind} \vnorm{\frac{z_k^\ind-z_{k-1}^\ind}{\tau}}_\ZZ^2 \dd r 
= \frac{1}{\tau}  \sum_{k=1}^N \norm{z_k^\ind-z_{k-1}^\ind}_\ZZ^2 .
\end{multline*}
Lemma~\ref{prop.finiteNumberSteps}, precisely estimate \eqref{eq:eq.est_z_H1}, thus implies that this term is bounded independent of $\tau$ and $\delta$, which proves the desired $H^1(0,\tilde{S};\ZZ)$ estimate.
\end{proof}

Eventually, we are now in the position to show a discrete version of the energy equality. 
Its proof is based on Lemma~\ref{lem.prop-affine-interpol}, Lemma~\ref{lem:unibound} and the a priori estimates derived in Section~\ref{sec:apriori}. 

\begin{Lemma}[Discrete energy equality]\label{prop:discrete-eneq}
For all $s \in [0,S_\ind]$, it holds
\begin{equation}\label{eq:eneq+rem}
\begin{aligned}
 & \II(\hat{t}_\ind(s),\hat{z}_\ind(s)) \\
 & \quad  + \int_{0}^{s} \RR_\delta(\hat{z}^\prime_\ind(\sigma)) 
 + \dist_{\VV^\ast}\{- D_z\II(\ubar{t}_\ind(\sigma),\ovbar{z}_\ind(\sigma)), \partial \RR_\delta (0)\} \dd \sigma \\
 &= \II(\hat{t}_\ind(0),\hat{z}_\ind(0)) \\
 &\quad + \int_{0}^{s} \partial_t \II(\hat{t}_\ind(\sigma),\hat{z}_\ind(\sigma))\, \hat{t}^\prime_\ind(\sigma) \dd \sigma 
 + \int_{0}^{s} r_\ind(\sigma) \dd \sigma \, ,
\end{aligned}
\end{equation}
where 
\begin{equation}\label{eq:defremainder}
 r_\ind(s) := 
 \dual{D_z\II(\hat{t}_\ind(s),\hat{z}_\ind(s)) - D_z\II(\ubar{t}_\ind(s),\ovbar{z}_\ind(s))}{\hat{z}_\ind^\prime(s)}_{\ZZ^*, \ZZ} \, .
\end{equation}
Moreover, the complementarity condition 
\begin{equation}
 \hat{t}_\ind^\prime(s) \dist_{\VV^\ast}\{- D_z\II(\ubar{t}_\ind(s),\ovbar{z}_\ind(s)),\partial\RR_\delta(0)\} = 0 
\label{eq:comp.cond}
\end{equation}
is fulfilled f.a.a. $s\in (0,S_\ind)$, and there exists a constant $C>0$ such that the remainder $r_\ind$ satisfies 
for all $\tau,\delta >0$ and all $s \in [0,S_\ind]$ 
\begin{equation}
 \int_{0}^{s} r_\ind(\sigma) \dd \sigma \leq C \tau .
\label{eq:remainder}
\end{equation}
\end{Lemma}
\begin{proof}
The complementarity in \eqref{eq:comp.cond} has already been proven in Lemma~\ref{lem.prop-affine-interpol}. Hence, we turn to the discrete energy identity. 
Since the affine interpolants in \eqref{eq:affine-interpolants} are by construction elements of $W^{1,\infty}(0,S_\ind)$ and 
$W^{1,\infty}(0,S_\ind;\ZZ)$, respectively, and due to $\II \in C^1([0,T] \times \ZZ)$ by assumption, 
the chain rule is applicable and gives for $s \in (s_{k-1}^\ind,s_k^\ind)$ that
\begin{align*}
 & \frac{\textup{d}}{\textup{d}s} \II(\hat{t}_\ind(s),\hat{z}_\ind(s)) \\
 &\qquad =\partial_t \II(\hat{t}_\ind(s),\hat{z}_\ind(s))\, \hat{t}^\prime_\ind(s) 
 + \dual{D_z\II(\hat{t}_\ind(s),\hat{z}_\ind(s))}{\hat{z}_\ind^\prime(s)}_{\ZZ^\ast,\ZZ} \\
 &\qquad =\partial_t \II(\hat{t}_\ind(s),\hat{z}_\ind(s))\, \hat{t}^\prime_\ind(s) 
 + \frac{1}{\tau}\dual{D_z\II(\ubar{t}_\ind(s),\ovbar{z}_\ind(s))}{\zth{k}-\zth{k-1}}_{\ZZ^\ast,\ZZ} \\
 &\qquad \quad + \dual{D_z\II(\hat{t}_\ind(s),\hat{z}_\ind(s)) 
 - D_z\II(\ubar{t}_\ind(s),\ovbar{z}_\ind(s))}{\hat{z}_\ind^\prime(s)}_{\ZZ^\ast,\ZZ} \, .
\end{align*} 
From \eqref{eq:opt.prop03}, we have in combination with the 1-homogeneity of $\RR_\delta$ that
\begin{align*}
 & - \frac{1}{\tau} \dual{D_z\II(\ubar{t}_\ind(s),\ovbar{z}_\ind(s))}{\zth{k}-\zth{k-1}}_{\ZZ^\ast,\ZZ} \\
 &\qquad = \frac{1}{\tau} \left(\RR_\delta(\zth{k}-\zth{k-1}) 
 + \tau \dist_{\VV^\ast}\{- D_z\Ih{k-1}{k},\partial \RR_\delta(0)\}\right) \\
 &\qquad = \RR_\delta(\hat{z}^\prime_\ind) + \dist_{\VV^\ast}\{- D_z\Ih{k-1}{k},\partial \RR_\delta(0)\}).
\end{align*}
By taking into account the definition of $r_\ind$ in \eqref{eq:defremainder}, 
integration over $(\sigma_1,\sigma_2)$ then yields \eqref{eq:eneq+rem}.\\
\noindent
It remains to estimate $r_\ind$. To this end, first observe that the definition of the affine and constant interpolants in 
\eqref{eq:affine-interpolants} and \eqref{eq:constinter} implies for every $k\in \{1, ..., N\}$ and every 
$s \in [s_{k-1}^\ind,s_k^\ind)$ that
\[ \hat{z}_\ind(s) - \ovbar{z}_\ind(s) = (s -s_{k}^\ind) \hat{z}_\ind^\prime(s) \, 
\text{ and } \, \hat{t}_\ind(s) - \ubar{t}_\ind(s) = (s -s_{k-1}^\ind) \hat{t}_\ind^\prime(s) , \]
which is frequently used in the following estimates.
Now, let $k \in \{1, ..., N\}$ and $s\in [s_{k-1}^\ind,s_k^\ind)$ be arbitrary. Then, since $(s-s_k^\ind )< 0$ the Garding-like inequality from \eqref{eq:garding-like} implies
\begin{equation}\label{eq:remainderest}
\begin{aligned}
 &\hspace*{-0.2cm}\dual{D_z\II(\hat{t}_\ind(s),\hat{z}_\ind(s)) - D_z\II(\ubar{t}_\ind(s),\ovbar{z}_\ind(s))}{\hat{z}_\ind^\prime(s)}_{\ZZ^*, \ZZ} \\
 &= \frac{1}{s-s_k^\ind} \dual{D_z\II(\hat{t}_\ind(s),\hat{z}_\ind(s)) - D_z\II(\hat{t}_\ind(s),\ovbar{z}_\ind(s))}{\hat{z}_\ind(s)-\ovbar{z}_\ind(s)}_{\ZZ^*, \ZZ} \\
 &\quad + \frac{1}{s-s_k^\ind} \dual{D_z\II(\hat{t}_\ind(s),\ovbar{z}_\ind(s)) - D_z\II(\ubar{t}_\ind(s),\ovbar{z}_\ind(s))}{\hat{z}_\ind(s)-\ovbar{z}_\ind(s)}_{\ZZ^*, \ZZ} \\
 &\leq \frac{1}{\abs{s-s_k^\ind}} \big( - \frac{\alpha}{2} \norm{\hat{z}_\ind(s)-\ovbar{z}_\ind(s)}_\ZZ^2 + C_1 \, \norm{\hat{z}_\ind(s)-\ovbar{z}_\ind(s)}_\VV^2 \big) \\
 &\quad + \frac{1}{\abs{s-s_k^\ind}} \big( C_2 \abs{\hat{t}_\ind(s)-\ubar{t}_\ind(s)} \norm{\hat{z}_\ind(s)-\ovbar{z}_\ind(s)}_\ZZ \big) \\
 &\leq \abs{s-s_k^\ind} \big( C_1 \, \norm{\hat{z}^\prime_\ind(s)}_\VV^2 \big) + C_2 \abs{\hat{t}_\ind(s)-\ubar{t}_\ind(s)} \norm{\hat{z}^\prime_\ind(s)}_\ZZ .
\end{aligned}
\end{equation}
Now, since $ \abs{s-s_k^\ind},\abs{\hat{t}_\ind(s)-\ubar{t}_\ind(s)} \leq \tau $ we obtain from the identity in \eqref{eq:defremainder} that  
\[  r_\ind(s) \leq \tau \, \big( C_1 \, \norm{\hat{z}^\prime_\ind(s)}_\VV^2 + C_2 \norm{\hat{z}^\prime_\ind(s)}_\ZZ ) \]
for almost all $ s \in [0,S_\ind]$. Thus \eqref{eq:remainder} easily follows from the bounds in Lemma~\ref{lem:unibound}. 
\end{proof}

\begin{Remark}
 A comparison of the discrete energy identity in \eqref{eq:eneq+rem} and the continuous one in \eqref{eq:energyiddef}
 shows that the coefficient $\norm{\hat{z}_\ind^\prime}$ is missing in front of the distance.
 It would be possible to reformulate the optimality conditions in Lemma~\ref{prop.optimalityprops} 
 in a way such that this coefficient would arise in \eqref{eq:eneq+rem}. 
 This, however, would complicate the passage to the limit in the next section. 
 As we will see at the end of the proof of Theorem~\ref{thm:main-conv}, \eqref{eq:eneq+rem} is 
 sufficient to obtain the desired energy identity in \eqref{eq:energyiddef}.
\end{Remark}

\subsection{Main Convergence Theorem}\label{sec:main}
Before we come to the main result, i.e., the passage to the limit in the discrete energy identity and therewith ultimately the existence of parametrized solutions, we need one last preparatory result, which guarantees the weak lower semicontinuity of the distance term in \eqref{eq:eneq+rem}.
\begin{Lemma}\label{lem:dist_lsc}
Let $ \xi_\delta \in \ZZ^\ast$ with $ \xi_\delta \weakly \xi $ in $\ZZ^\ast$ for $ \delta \to 0$. Suppose, moreover, that the distance is uniformly bounded, i.e., $ \dist_{\VV^\ast}\{-\xi_\delta,\partial \RR_\delta(0)\} \leq C$ with $C$ independent of $\delta$. Then the following weak lower semicontinuity result holds true:
\begin{equation}\label{eq:dist_lsc}
\liminf_{\delta \to 0} \dist_{\VV^\ast}\{-\xi_\delta,\partial\RR_\delta(0)\} \geq \dist_{\VV^\ast}\{-\xi,\partial\RR(0)\} .
\end{equation}
\end{Lemma}
\begin{proof}
First of all, we know that the minimum in the definition of the distance is attained, cf. Lemma~\ref{lem:conv.ana}, so that there exists $ \mu_\delta \in \partial\RR_\delta(0) \subset \ZZ^*$ with 
\begin{equation}
 \dist_{\VV^\ast}\{-\xi_\delta,\partial\RR_\delta(0)\} = \norm{\mu_\delta+\xi_\delta}_{\VV^*} .
\end{equation}
Therewith, we define $ \eta_\delta := \mu_\delta+\xi_\delta$ and infer $ \norm{\eta_\delta}_{\VV^\ast} \leq C$ by assumption. Hence, we may extract a weakly convergent subsequence $\eta_{\delta_n} \weakly \eta $ in $\VV^\ast$ for $ n \to \infty$. In particular, due to the lower semicontinuity of the norm $\norm{\cdot}_{\VV^*}$, it holds
\begin{equation}\label{eq:aux.liminf_eta}
\norm{\eta}_{\VV^*} \leq \liminf_{n \to \infty} \norm{\eta_{\delta_n}}_{\VV^*} = \liminf_{n \to \infty} \dist_{\VV^\ast}\{-\xi_{\delta_n},\partial \RR_{\delta_n}(0)\} .
\end{equation}   
We proceed with showing that $ \eta = \mu+\xi$ for some $ \mu \in \partial\RR(0)$. To this end, we first note that by $\VV^\ast \subset \ZZ^\ast$ and the weak convergence of $\xi_{\delta_n}$ it holds $ \mu_{\delta_n} = \xi_{\delta_n} - \eta_{\delta_n} \weakly \xi - \eta$ in $\ZZ^\ast$ and we define $ \mu = \xi - \eta$.
Now, $\mu_{\delta_n} \in \partial \RR_{\delta_n}(0)$ is equivalent to
\begin{equation*}
 \RR_{\delta_n}(z) \geq \dual{\mu_{\delta_n}}{z}_{\ZZ^\ast,\ZZ} \quad \forall\, z \in \ZZ
\end{equation*}
and this inequality remains in the limit $ n \to \infty$. Indeed, given $ z \in \ZZ$, by assumption \eqref{ass:Gamma-Conv-R} there exists a sequence $z_n \in \ZZ$ converging to $z$ with $\RR(z) \geq \limsup_{n \to \infty} \RR_{\delta_n}(z_n)$. The strong convergence of $z_n$ also implies that the dual pairing $\dual{\mu_{\delta_n}}{z}_{\ZZ^\ast,\ZZ}$ converges so that 
\begin{equation}
\RR(z) \geq \limsup_{n \to \infty} \RR_{\delta_n}(z_n) \geq \limsup_{n \to \infty} \dual{\mu_{\delta_n}}{z_n}_{\ZZ^\ast,\ZZ} = \dual{\mu}{z}_{\ZZ^*,\ZZ}.
\end{equation}
Since $z \in \ZZ$ was arbitrary, we find $\mu \in \partial\RR(0)$. Hence, we conclude from \eqref{eq:aux.liminf_eta} that
\begin{align*}
\dist_{\VV^\ast}\{-\xi,\partial\RR(0)\} \leq \norm{\mu+\xi}_{\VV^*} = \norm{\eta}_{\VV^*} \leq \liminf_{n \to \infty} \dist_{\VV^\ast}\{-\xi_{\delta_n},\partial \RR_{\delta_n}(0)\} . 
\end{align*}
Since this holds for all subsequence of $ \eta_\delta$, we ultimately arrive at the desired lower semicontinuity in \eqref{eq:dist_lsc}.
\end{proof}

\begin{Example}
Note that it is indeed possible that $ \xi \in \ZZ^* \setminus \VV^*$ while $ \dist_{\VV^\ast}\{-\xi,\partial\RR(0)\} < \infty$. To see this, let us take $ \ZZ = H^1(0,1)$, $ \VV = L^2(0,1)$ and $ \RR(v) = \norm{v}_{L^1(0,1)} + I_K(v)$ where $ K = \{ v \in H^1(0,1) : v \geq 0 \text{ a.e. in } (0,1) \}$. Moreover, we let $ \xi = \delta_{1/2} \in H^{-1}(0,1)$ with $\delta_{1/2}$ the delta distribution in $ x = \tfrac{1}{2}$, i.e., $ \dual{\xi}{z}_{\ZZ^*,\ZZ} = z(\tfrac{1}{2})$. Due to the construction, we have $ \RR(v) \geq 0 \geq -v(\frac{1}{2}) = \dual{-\xi}{v}_{\ZZ^*,\ZZ}$ for all $ v \in K$ and since $\RR(v) = \infty$ for $ v \not\in K$ also $\RR(v) \geq \dual{-\xi}{v}_{\ZZ^*,\ZZ}$ for all $ v \in \ZZ$. Therefore, by the characterization of $\partial\RR(0)$ in Lemma~\ref{lem:char.subdiff}, it holds $ -\xi \in \partial\RR(0)$ and consequently $ \dist_{\VV^\ast}\{-\xi,\partial\RR(0)\} = 0 $ although $ \xi \not\in \VV^*$. Clearly, this property is related to the unboundedness of $\RR$, i.e., $\RR$ does not fulfill the upper bound $\RR(v) \leq C \norm{v}_\VV$, which is a frequently used assumption in the context of parametrized solutions.
\end{Example}

We now have everything at hand to prove our main convergence result.

\begin{theorem}[Convergence towards parametrized solutions]\label{thm:main-conv}
Assume that $ z_0^\delta $ converges to the initial state $z_0$ for $\delta \to 0$. Then there exists a sequence of parameters $\{\tau_n,\delta_n\}_{n\in \N}\subset \R_+ \times \R_+$ converging to zero so that the affine interpolants generated 
by the fully discrete local stationarity scheme (\nameref{alg:locmin}) and the artificial end time defined in \eqref{eq:Sdef} satisfy
\begin{alignat}{3}
 S_{\tau_n,\delta_n} &\to S, \label{eq:conv.S} \\
 \hat{t}_{\tau_n,\delta_n} &\overset{\ast}{\weakly} \hat{t} & \quad & \text{ in }\; W^{1,\infty}(0,S;\R), \label{eq:conv.t} \\
 \hat{z}_{\tau_n,\delta_n} &\overset{\ast}{\weakly} \hat{z} & \quad & \text{ in }\; W^{1,\infty}(0,S;\VV) \cap H^1(0,S;\ZZ), 
 \label{eq:conv.z} \\
 \hat{z}_{\tau_n,\delta_n}(s) &\weakly \hat{z}(s) & \quad & \text{ in }\; \ZZ \text{ for every } s \in [0,S], \label{eq:conv.z_ptw}
\end{alignat}
and the limit $(\hat t, \hat z)$ is a parametrized solution in the sense of Definition~\ref{def:paramsol}.

Moreover, every accumulation point $(\hat{t},\hat{z})$ of sequences in the sense of 
\eqref{eq:conv.S}--\eqref{eq:conv.z_ptw} is a parametrized solution.
\end{theorem}

\begin{proof}
The arguments are analog to the ones in \cite{fem_paramsol,diss_sievers}. For convenience of the reader, we briefly repeat the main steps.  

\noindent
The existence of a (sub-)sequence satisfying \eqref{eq:conv.S}--\eqref{eq:conv.z} 
is an immediate consequence of the uniform estimates in Lemma~\ref{lem.basicest1}, Lemma~\ref{lem:unibound}, 
and \eqref{eq:Sdef}. The pointwise convergence in \eqref{eq:conv.z_ptw} follows from the Aubin-Lions lemma, i.e., $W^{1,\infty}(0,S;\VV) \cap L^\infty(0,S;\ZZ) \embeds^c C(0,S;\VV)$, the density of $\ZZ $ in $\VV$ and the fact that for every $s\in [0,S]$, $\{\hat{z}_{\tau_n,\delta_n}(s)\}_{n\in\N}$ is bounded in $\ZZ$ by Lemma~\ref{lem.basicest1}.

\noindent
It remains to show that every (weak) limit is a parametrized solution. 
For this purpose, let $\{\tau_n,\delta_n\}$ be an arbitrary null sequence and assume that the convergences in  
\eqref{eq:conv.S}--\eqref{eq:conv.z_ptw} hold. In order to simplify the notation, we indicate by $\{\cdot\}_n$ 
the sequence of $\{\cdot\}_\ind$ corresponding to $\{\tau_n,\delta_n\}$. 
Analogously, we abbreviate the index $\delta_n$ simply by $n$.
We proceed in several steps and start with the following: \vspace*{0.1cm}

\indent \textbf{Convergence of piecewise constant interpolants.}
One easiy verifies using the estimate 
\begin{equation*}
 \|\hat z_n(s) - \ovbar{z}_n(s)\|_\VV = |s - s_k^n|\,\|\hat z_n'(s)\|_\VV \leq \tau \to 0, 
\end{equation*}
which holds for all $k\in \{1, ..., N\}$ and all $s\in [s_{k-1}^n, s_k^n)$ that the piecewise constant interpolants converge 
pointwise to the same limit as the affine interpolants. We therefore have 
\begin{equation}\label{eq:convconst}
 \ubar{t}_n(s), \ovbar{t}_n(s) \to \hat{t}(s), \qquad \ubar{z}_n(s), \ovbar{z}_n(s) \weakly \hat{z}(s) \;\text{ in } \ZZ 
 \quad \forall\, s\in [0,S].
\end{equation}

\indent \textbf{Initial and end time conditions.}
By assumption we have $ \hat{z}_n(0) = z_0^\delta \to z_0 $ in $\ZZ$,
so that the pointwise convergence in \eqref{eq:conv.z_ptw} implies $\hat z(0) = z_0$ as desired. 
Moreover, thanks to \eqref{eq:conv.t}, $\hat t_n$ converges uniformly to $\hat t$ so that 
\begin{equation*}
 0 = \hat t_n(0) \to \hat t(0) 
 \quad \text{and} \quad 
 T = \hat t_n(S_n) \to \hat t(S), 
\end{equation*}
where we also used \eqref{eq:conv.S}.

\indent \textbf{Complementarity relations.}
We continue with the complementarity-like relations in \eqref{eq:compl}. 
First, the set 
\begin{equation*}
 \{ (\tau, \zeta) \in L^2(0,S) \times L^2(0,S;\VV) : \tau(s) \geq 0, \; 
 \tau(s) + \|\zeta(s)\|_\VV \leq 1 \text{ f.a.a.\ } s\in (0,S)\}
\end{equation*}
is clearly convex and closed, thus weakly closed and consequently, 
we obtain that the weak limit $(\hat t, \hat z)$ satisfies the inequalities in \eqref{eq:degenerate}. Next, we turn to \eqref{eq:slack}, whose derivation is by far more involved.
On account of the weak continuity assumptions for $D_z\II$, it follows 
from \eqref{eq:convconst} that 
\[ D_z\II(\ubar{t}_n(s),\ovbar{z}_n(s)) \weakly D_z \II(\hat{t}(s),\hat{z}(s)) \;\text{ in } \ZZ^\ast
\quad \forall\, s\in [0,S]. \]
Combining this with the uniform boundedness of the distance from \eqref{eq:eq.est_dist} allow us to apply Lemma~\ref{lem:dist_lsc}, which gives
\begin{equation}\label{eq:liminf1}
\begin{aligned}
 \liminf_{n \to \infty} \dist_{\VV^\ast}\{- D_z\II(\ubar{t}_n(s),\ovbar{z}_n(s)),\partial \RR_n(0)\} 
 \qquad\qquad &\\
 \geq  \dist_{\VV^\ast}\{-D_z\II(\hat{t}(s),\hat{z}(s)),\partial \RR(0)\}. &
\end{aligned}
\end{equation}
To show \eqref{eq:slack}, let us abbreviate
\begin{equation*}
\begin{aligned}
 \xi_n(s) &:= \dist_{\VV^\ast}\{- D_z\II(\ubar{t}_n(s),\ovbar{z}_n(s)),\partial \RR_n(0)\},\\
 \xi(s) &:= \dist_{\VV^\ast}\{-D_z\II(\hat{t}(s),\hat{z}(s)),\partial \RR(0)\},
\end{aligned}
\end{equation*}
so that \eqref{eq:liminf1} reads
\begin{equation}\label{eq:liminf2}
 \liminf_{n \to \infty} \xi_n(s) \geq \xi(s) \geq 0 \quad \forall\, s\in [0,S].
\end{equation}
Concerning the measurability of $\xi$ we note that by the embedding $H^1(0,S;\ZZ) \embeds C(0,T;\ZZ)$ and the continuity of $D_z\II$ the mapping $ s \mapsto -D_z\II(\hat{t}(s),\hat{z}(s))$ is continuous. Exploiting Lemma~\ref{lem:dist_lsc}, we can conclude that $\xi$ is lower semicontinuous and therefore, indeed, measurable.

\noindent
Now, consider an arbitrary $\omega \geq 0$ and define $\xi_{n,\omega}(s) := \min\{\xi_n(s),\xi(s),\omega\}$ 
such that, thanks to \eqref{eq:liminf2}, $\xi_{n,\omega}(s)$ converges to $\xi_\omega(s) := \min\{\xi(s),\omega\}$ almost everywhere in $(0,S)$.
Since $\xi_\omega$ is measurable (as $\xi$ is so) and $\omega \geq \xi_{n,\omega}(s)$, Lebesgue's dominated convergence theorem gives $ \xi_{n,\omega} \to \xi_{\omega} $ in $ L^1(0,S)$. Thus, thanks to $\xi_n(s) \geq \xi_{n,\omega}(s)$ and the weak$^*$ convergence of $\hat t'$, we obtain from \eqref{eq:comp.cond} that
\begin{align*}
 0 =  \liminf_{n \to \infty} \int_0^S \hat{t}^\prime_n(s) \, \xi_n(s)\dd s 
 \geq \liminf_{n \to \infty}  \int_0^S \hat{t}^\prime_n(s) \, \xi_{n,\omega}(s) \dd s
 = \int_0^S \hat{t}^\prime(s) \, \xi_\omega(s) \dd s  .
\end{align*}
Since $\omega\geq 0$ was arbitrary, this inequality holds for every $\omega$ so that Fatou's lemma yields
\begin{align*}
0 \geq \liminf_{\omega \to \infty} \int_0^S \hat{t}^\prime(s)\, \xi_\omega(s) \dd s 
\geq \int_0^S \hat{t}^\prime(s) \, \xi(s) \dd s \geq 0.
\end{align*}
Because of $\xi \geq 0$ and $\hat t' \geq 0$ a.e.\ in $(0,S)$, cf.~\eqref{eq:degenerate}, this gives \eqref{eq:slack}.

\indent \textbf{Energy identity.}
The energy identity is a direct consequence of its discrete version in Lemma~\ref{prop:discrete-eneq}. Indeed, we find
\begin{align*}
 & \II(\hat{t}(s),\hat{z}(s)) + \int_0^{s} \RR(\hat{z}^\prime(\sigma)) 
 + \norm{\hat{z}^\prime(\sigma)}_\VV \dist_{\VV^\ast}\{-D_z\II(\hat{t}(\sigma),\hat{z}(\sigma)),\partial \RR(0)\} \dd \sigma \\
 &\leq \liminf_{n \to \infty} \bigg( \II(\hat{t}_n(s),\hat{z}_n(s)) \\
 &\qquad  + \int_0^{s} \RR_n(\hat{z}_n^\prime(\sigma)) 
 + \dist_{\VV^\ast}\{- D_z\II(\hat{t}_n(\sigma),\hat{z}_n(\sigma)),\partial \RR_n(0)\} \dd \sigma \bigg) \\
 & = \liminf_{n \to \infty} \left( \II(\hat{t}_n(0),\hat{z}_n(0)) 
 + \int_0^{s} \partial_t \II(\hat{t}_n(\sigma),\hat{z}_n(\sigma)) \, \hat{t}^\prime_n(\sigma) \dd \sigma + \int_0^{s} r_n(\sigma) \dd \sigma \right).
\end{align*}
by the weak lower semicontinuity of $\II(t, \cdot)$ from \eqref{eq:lower_semicont_combined} and \cite[Cor. 4.5]{stefanelli2008brezis} which gives 
\[ \int_0^s \RR(\hat{z}^\prime(\sigma)) \dd \sigma \leq \liminf_{\delta \to 0} \int_0^s \RR_\delta(\hat{z}^\prime(\sigma)) \dd \sigma \]
due to the assumptions on the space $\VV$ and condition \eqref{ass:Gamma-Conv-R} on $\RR_\delta$. Moreover, we used $\|\hat z'(s)\|_\VV \leq 1$ as well as Fatou's lemma together with \eqref{eq:liminf1} for the distance term. Now, exploiting the estimate in \eqref{eq:remainder}, assumption \ref{ass:E4} for $\partial_t \II$ and the strong convergence of $\hat{z}_n(0)$ to $z_0$ in $\ZZ$ we finally end up with
\begin{multline}\label{eq:energy_ineq_main_conv}
\II(\hat{t}(s),\hat{z}(s)) + \int_0^{s} \RR(\hat{z}^\prime(\sigma)) 
 + \norm{\hat{z}^\prime(\sigma)}_\V \dist_{\VV^\ast}\{-D_z\II(\hat{t}(\sigma),\hat{z}(\sigma)),\partial \RR(0)\} \dd \sigma \\
\leq \II(0,z_0) + \int_0^{s} \partial_t \II(\hat{t}(\sigma),\hat{z}(\sigma)) \, \hat{t}^\prime(\sigma) \dd \sigma,
\end{multline}
which is the desired energy inequality. Taking into account that $ \hat{z} \in H^1(0,S;\ZZ)$, it is well-known that the sole inequality \eqref{eq:energy_ineq_main_conv} is already equivalent to the energy identity \eqref{eq:energyiddef}, see \citep[Lem. 6.6]{krz13} or \citep[Lem. 2.4.6]{diss_sievers}, which completes the proof.
\end{proof}
Unfortunately, we do not obtain the nondegeneracy let alone normalization of the limit $(\hat{t},\hat{z})$ here. The main problem is the fact that the weak convergence of $\hat{z}_n$ in $H^1(0,S;\ZZ)$ from \eqref{eq:conv.z} is not sufficient in order to pass to the limit in \eqref{eq:affine.aux1}, that is, $\hat{t}_n^\prime(s) + \norm{\hat{z}_n^\prime(s)}_\VV = 1$, and still obtain equality in the end. In \cite{mielkezelik,efenmielke06}, the authors therefore provide sufficient conditions, which guarantee the nondegeneracy of the limit function. Moreover, in \cite{efenmielke06}, a condition is given, which also preserve the normalization. Nevertheless it is always possible to reparameterize a parametrized solution and in order to normalize it, see \citep[Lerm. A.4.3]{diss_sievers}. Regardless of this fact, we note that the above Theorem, while dedicated to the convergence analysis of the fully discrete local stationarity scheme, also provides an existence result for parametrized solutions in case of an unbounded dissipation $\RR$ (choose $ \RR_\delta = \RR$).

\section{Application to a Damage Model}\label{sec:num}
We now aim at applying the local stationarity scheme to model the evolution of damage within a workpiece during a time interval $[0,T]$. For this, we let $\Omega \subset \R^2$ be a bounded domain that corresponds to an elastic body and satisfies
\begin{center}
$ \Omega \subset \R^2 $ has a Lipschitz boundary $ \partial \Omega = \overline{\Gamma}_D \cup \overline{\Gamma}_N$ with Dirichlet boundary $ \Gamma_D $ \\
such that $\HH^1(\Gamma_D) > 0$ and Neumann boundary $\Gamma_N$. Moreover, $ \Gamma_D$ and $\Gamma_N$ are \\
supposed to be regular in the sense of Gröger, see \cite{groger89}. 
\end{center} 
Note that, although we focus on the twodimensional case here, it is also possible to consider the threedimensional case aswell provided the spaces and the energy are adapted appropriately; compare with the elaborations in \cite{livia2019part2,krz13} which we also follow with regard to notation. During the time $[0,T]$, time dependent boundary conditions $u_D$ as well as external boundary and volume forces $\ell$ may be applied, which lead to a certain displacement $u$ and possibly even to a damage, represented by the variable $z$, of the body. Usually, $z$ is supposed to take values in $[0,1]$ whereby $ z(t,x)=0$ means the body is completely sound and, correspondingly, $ z(t,x)=1$ means the body is comletely damaged. With a view to the energy functional, we define \[ \UU = \{ v \in H^1(\Omega,\R^2) \, : \, v|_{\Gamma_D} = 0 \} , \quad \ZZ=H^1(\Omega), \quad \VV = L^2(\Omega) \]
and let 
\[ \ell \in C^{1,1}(0,T;W^{-1,p}_{\Gamma_D}(\Omega)), \quad u_D \in C^{1,1}(0,T;W^{1,p}(\Omega)) \]
where $ p > 2 $ is chosen as in Lemma~\ref{lem:uniquely_solvable} below. Note that we will use a scaled version of the $L^2$-norm, that is $\norm{\cdot}_\VV := \tfrac{1}{\abs{\Omega}}\norm{\cdot}_{L^2}$. This choice has been shown to be advantageous in the numerical experiments particularly with a view to iteration numbers. Now, we set $\EE : [0,T] \times \UU \times \ZZ \to \R$ as
\begin{align*} 
\EE(t,u,z) &= \frac{1}{2}\int_\Omega \abs{\nabla z}^2 \dd x + \int_\Omega f(z) \dd x \\ 
&\qquad + \frac{1}{2} \int_\Omega g(z) \, \C \varepsilon(u+u_D(t)):\varepsilon(u+u_D(t)) \dd x - \dual{\ell(t)}{u}_\UU \\
&= \II_1(z) + \EE_2(t,u,z). 
\end{align*}
where $ \C $ is the usual elasticity tensor with
\begin{subequations}\label{eq:elasticity_tensor}
\begin{align} 
&\C \in L^\infty(\Omega; \LL(\R_{sym}^{d \times d}, \R_{sym}^{d \times d}) \\
&\exists \gamma_0 > 0  \text{ such that for all } \xi \in \R_{sym}^{d \times d} \text{ and for almost all } x \in \Omega: \quad    \C(x) \xi : \xi \geq \gamma_0 \norm{\xi}^2 
\end{align}
\end{subequations}
and $ \varepsilon(u) = \tfrac{1}{2}(\nabla u + \nabla u^\top)$ is the linearized strain tensor. The nonlinearity $f$ in the energy is chosen such that the term is continuously differentiable in $\ZZ$. Since we will neglect this term in our numerical examples, we do not particularize the exact assumptions and refer to \cite{krz13}. However, a typical choice for $f$ in the context of Ambrosio-Tortorelli approximation of brittle fracture is $f(z) = (1-z)^2$, see \cite{giacomini05,almi2019consistent}, which is clearly sufficiently smooth. Furthermore, the function $g$, which somehow represents the preservation of the elasticity of the material depending on the state of damage, is supposed to fulfill:
\begin{equation}\label{eq:func_g}
g \in C^2(\R), \text{ with } g^\prime,g^{\prime \prime} \in L^\infty(\R), \text{ and } \exists \gamma_1,\gamma_2 >0 \, : \, \forall z \in \R \, : \, \gamma_1 \leq g(z) \leq \gamma_2. 
\end{equation}
In particular, the lower bound $ g \geq \gamma_1 >0$ is to be noted here. It implies that even if the material is completely damaged, it does not lose all its rigidity. This is often referred to as \emph{partial damage model}. Finally, the dissipation $\RR: L^1(\Omega) \to [0,\infty]$ is given by
\begin{equation}\label{eq:diss_damage} \RR(v) = \begin{cases}
\kappa \int_\Omega v(x) \dd x, & \text{if } v \geq 0 \, \text{ a.e. in } \Omega, \\
+ \infty, & \text{else}, 
\end{cases} \end{equation}
with the so-called fracture toughness $\kappa >0$.
Now, in order to bring this model into the setting of Section~\ref{sec:assumptions}, it is convenient to reduce the system to the damage variable $z$. This means that we require the displacement $u(t)$ to minimize the energy $\EE(t,\cdot,z(t))$ at every time point $ t \in [0,T]$, i.e., 
\begin{equation}\label{eq:damage_min_u} 
u(t) \in \argmin \{ \EE(t,v,z) \, : \, v \in \UU \} . 
\end{equation}
It is, in fact, easy to see that this problem has a unique minimizer for every $ t \in [0,T]$ and $ z \in \ZZ$. Hence, we define $ \II_2 : [0,T] \times \ZZ \to \R$ by $ \II_2(t,z) = \inf_{v \in \UU} \EE_2(t,v,z)$ and let 
\begin{equation}\label{eq:energy_damage} \II(t,z) = \II_1(z) + \II_2(t,z). \end{equation} 

\subsection{Properties of the energy functional}\label{sec:energy_props}
We now want to verify that the model from above fits into the setting of Section~\ref{sec:assumptions}. Thereby, we rely on the results from \cite{livia2019part2,krz13}. We start with the following observation, which was first proven in \cite{hmw11} and states that the minimization with respect to $u$ is well-defined and provides a unique solution. 
\begin{Lemma}\label{lem:uniquely_solvable}
Under the assumptions \eqref{eq:elasticity_tensor} and \eqref{eq:func_g} there exists $ p > 2 $ such that for any $ \tilde{p} \in [0,p]$ and for every $ z \in \ZZ$ the linear elliptic operator 
\[ \dual{L_z(v)}{w} = \frac{1}{2} \int_\Omega g(z) \, \C \varepsilon(v):\varepsilon(w) \dd x \quad \forall v,w \in \UU \]
is an isomorphism $ L_z : W^{1,\tilde{p}}_{\Gamma_D}(\Omega;\R^d) \to W^{-1,\tilde{p}}_{\Gamma_D}(\Omega;\R^d)$.
\end{Lemma}
Therefore, the reduced energy $ \II_2(t,z) = \inf_{v \in \UU} \EE_2(t,v,z)$ is also well-defined and we can focus on the properties of this part in the overall energy $\II$. 
\begin{Lemma}\label{lem:est_damage_energy}
Let $ d=2$, $p>2$ and the assumptions \eqref{eq:elasticity_tensor}, \eqref{eq:func_g} hold. Then there exist constants $ C_1,C_2,c_3 > 0 $ such that 
\begin{gather}
\II_2(t,z) \geq - C_1 \quad \text{and} \quad \abs{\partial_t\II_2(t,z)} \leq C_2
\end{gather}
as well as 
\begin{equation}\label{eq:delta_DzI2}
\dual{D_z\II_2(t_1,z_1)-D_z\II_2(t_2,z_2)}{v}_{\ZZ^*,\ZZ} \leq c_3 ( \abs{t_2-t_1} + \norm{z_1-z_2}_{L^r(\Omega)} )  \norm{v}_{\ZZ} .
\end{equation}
for every $ r \in [\frac{6p}{p-4},\infty)$, where $p>2$ is as in Lemma~\ref{lem:uniquely_solvable}. Moreover, for any sequences $ t_k \to t$ and $ z_k \weakly z$ in $\ZZ$, it holds
\begin{gather}
D_z\II_2(t_k,z_k) \weakly D_z\II_2(t,z) \quad \text{in } \ZZ^\ast, \\
\II_2(t_k,z_k) \to \II_2(t,z) \quad \text{and} \quad \partial_t\II_2(t_k,z_k) \to \partial_t\II_2(t,z).
\end{gather}  
\end{Lemma}
\begin{proof}
This is a combination of Lemma~2.4, 2.6 and 2.8 as well as Corollary~2.9 from \citep{krz13}.
\end{proof}
With a view to Section~\ref{sec:num}, we set $\II_1(z) = \tfrac{1}{2}\dual{A z}{z}_{\ZZ^*,\ZZ}$ with $ A = - \laplace $. The above Lemma thus guarantees that $ \II(t,z) = \II_1(z) + \II_2(t,z)$ complies with the assumptions \ref{ass:E1} - \ref{ass:E4} and \eqref{eq:DzI-time-Lip} - \eqref{ass:I5} as well as the G{\aa}rding-like inequality \eqref{eq:garding-like}, i.e.
\[ \dual{D_z\II(t,z_1)-D_z\II(t,z_2)}{z_1-z_2}_{\ZZ^*,\ZZ} \geq \alpha \norm{z_1-z_2}_\ZZ^2 - \lambda\norm{z_1-z_2}^2_\VV. \] Indeed, we have the following:

\begin{theorem}
Let $ \II(t,z) = \II_1(z) + \II_2(t,z)$ be given as in \eqref{eq:energy_damage} with $\II_1(z) = \tfrac{1}{2}\dual{A z}{z}_{\ZZ^*,\ZZ}$ where $ A = - \laplace $. Moreover, let $ d=2$, $p>2$ and the assumptions \eqref{eq:elasticity_tensor}, \eqref{eq:func_g} hold. Then $\II$ fulfills \ref{ass:E1} - \ref{ass:E4} and \eqref{eq:DzI-time-Lip} - \eqref{ass:I5} as well as the G{\aa}rding-like inequality \eqref{eq:garding-like}. In particular, there exists at least one parametrized solution to the rate-independent system defined by $\II$ and $\RR$ as given \eqref{eq:energy_damage} and \eqref{eq:diss_damage}, respectively.  
\end{theorem}
\begin{proof}
The conditions \ref{ass:E1} - \ref{ass:E4} and \eqref{eq:DzI-time-Lip} - \eqref{ass:I5} follow immediately from the above Lemma~\ref{lem:est_damage_energy}. In addition, the G{\aa}rding-like inequality \eqref{eq:garding-like} is an easy consequence of the properties of $A$ and the inequality in \eqref{eq:delta_DzI2}. Thus, we see that $\II$ fulfills all assumptions from Section~\ref{sec:assumptions} so that applying Theorem~\ref{thm:main-conv} proofs the existence of a parametrized solution.
\end{proof}

As seen above Theorem~\ref{thm:main-conv} guarantees the existence of at least one parametrized solution for \eqref{subdiff-inclusion} in the setting of (partial) damage here. What is more, we may approximate such a solution by the local incremental stationarity scheme (\nameref{alg:locmin}), which is purpose of the following subsections.

\subsection{Finite Element discretization}\label{sec:FE}
As the convergence analysis from Theorem~\ref{thm:main-conv} allows us to use an approximation $\RR_\delta$ of the dissipation $ \RR$, we may use some finite element discretization to approximate a parametrized solution. Hence, we assume that a family $\{\TT_h\}_{h>0}$ of shape-regular triangulations of the domain $\Omega$ be given. Herein, $h$ denotes the mesh size defined by $h := \max_{T\in \TT_h} \diam(T)$. To keep the discussion concise, we also assume that $\Omega$ is a polygon and polyhedron, respectively, and that the triangulations exactly fit the boundary. For the discrete space, we choose the space of piecewise linear and continuous test functions, i.e., 
\begin{align*}
&\UU_h := \{ u \in C(\bar\Omega;\R^2) : u|_{T} \in \PP_1 \; \forall\, T \in \TT_h , \; u|_{\Gamma_D}= 0 \}. \\
\text{and} \quad &\ZZ_h := \{ v\in C(\bar\Omega) : v|_{T} \in \PP_1 \; \forall\, T \in \TT_h \}.
\end{align*}
In addition, we set 
\[ \RR_h(v) = \begin{cases} \RR(v), & v \in \ZZ_h \\ + \infty, & \text{else} \end{cases}. \]
By standard arguments, the lower inequality in \eqref{ass:Gamma-Conv-R} is satisfied for $ \delta = h \to 0$. For the upper inequality assume that $\RR(z) < \infty$, i.e., $z \geq 0 $ a.e. in $\Omega$ (otherwise there is nothing to show). Since $\Omega$ has a Lipschitz-boundary we can extend $z$ beyond $\Omega$ (cf. \citep[A8.12]{alt16en}) and use standard convolution in order to obtain an approximation $ z_k \in C^\infty(\overline{\Omega})$. By the construction of the extension and the non-negativity of the convolution kernel, $z_k$ is also non-negative. Moreover, we have $\norm{z_k-z}_\ZZ \to 0 $ for $ k \to \infty$. For $z_k \in C^\infty(\overline{\Omega})$ the classical, pointwise Lagrange-interpolation $ I_h $ is well-defined so that $ \norm{z_k - I_h(z_k)}_{\ZZ} \to 0 $ for $ h \to 0$. Hence, for any $ k \in \N$, there exists $h_k > 0 $ with $ \norm{z_k - I_{h}(z_k)}_{\ZZ} \leq 1/k$ for all $h \leq h_k$. W.l.o.g. we may assume that $ \{h_k\}_{k \in \N} $ is strictly monotonic decreasing. Therewith, we define $ z_h := I_{h}(z_k) $ for $ h_k \geq h > h_{k+1} $. Combining the above properties, we find that $ z_h \geq 0$ a.e. in $\Omega$ with $ z_h \to z$ in $\ZZ$ for $ h \to 0 $ and 
\[ \limsup_{h \to 0} \RR_{h}(z_h) = \lim_{h \to 0} \RR(z_h) = \RR(z) \]
by the continuity of $\RR$ on the set of non-negative functions in $\ZZ$. 


Before we proceed, let us set some notation. Given the triangulation $\TT_h$ the associated nodes and nodal basis are denoted by $x_i$ and $\varphi_i$, $i = 1, ..., N_h$, respectively. Moreover, given a function $z_h\in \ZZ_h$, we denote the coefficient 
vector of $z_h$ w.r.t.\ the nodal basis by $\bz = (z_1, ..., z_{N_h}) \in \R^{N_h}$, i.e., 
$z_h(x) = \sum_{i=1}^{N_h} z_i \,\varphi_i(x)$. Therewith, we may write 
\[  \int_\O z_h(x) \dd x = \bm^\top \bz 
 \quad \text{with} \quad \bm = (m_1, ..., m_{N_h}):= M \mathds{1}, \]
where $M_{ij} = \int_\Omega \varphi_i\,\varphi_j\,\d x \in \R^{N_h\times N_h}$ is the mass matrix and 
$\mathds{1} = (1, ..., 1) \in \R^{N_h}$. Thus, by the nonnegativity of $\varphi_i$, the discrete dissipation potential $\RR_h: \ZZ_h \to \R \cup \{ \infty \}$ can be written as
\begin{equation}
 \RR_h(z_h) := \left. \begin{cases} \kappa \, \bm^\top \bz, & z_i \geq 0 \; \forall i=1,...,N_h \\
 +\infty, & \text{else} \end{cases} \right\} = \kappa \bm^\top \bz + I_K(\bz),
 \label{eq:equiv_Rh}
\end{equation}
with the indicator functional $I_K$ corresponding to the cone $ K := \R^{N_h}_{\geq 0}$.

\subsection{Discrete energy functional}\label{sec:discr_energy}
In analogy to $\bz$ we denote by $ \bu = (u_{1,1}, \, u_{1,2},\, u_{2,1}, \, u_{2,2} ..., \, u_{N_h,2}) \in \R^{2 N_h}$ the coefficient vector corresponding to $ u_h \in \UU_h$. For the discrete version of the energy $\EE$ we use some slightly altered ansatz, namely
\begin{multline}
\tilde{\EE}(t,z_h,u_h) = \frac{1}{2} \int_\O | \nabla z_h |^2 \dd x \\
+ \frac{1}{2} \sum_{T \in \TT} \int_T g(z_h(x_{T}))\, \C \varepsilon(u_h+u_D(t)):\varepsilon(u_h+u_D(t)) \dd x - \dual{\ell(t)}{u_h}_\UU
\end{multline}
where $ x_T $ denotes the center of $ T$. It is clear that the minimization of $\tilde{\EE}$ with respect to $ u_h $ also provides a unique solution $ \bar{u}_h $ for every $ t \in [0,T]$ and $ z_h \in \ZZ_h$. Let $ S_h : \ZZ_h \to \UU_h$ denote the corresponding solution operator, i.e. it holds for all $ z_h \in \ZZ_h$ that
\begin{equation}\label{eq:discr_sol_operator}
\bar{u}_h = S_h(z_h) \quad \Leftrightarrow \quad D_u \tilde{\EE}(t,z_h,\bar{u}_h) = 0 . 
\end{equation}
With this operator at hand, we can also reduce the discrete energy $ \tilde{\EE} $ to the discrete damage variable $ z_h$. Hence, we define $\tilde{\II} : \ZZ_h \to \R$ as $ \tilde{\II}(t,z_h) = \tilde{\EE}(t,z_h,S_h(z_h))$. 
Altogether, with a little abuse of notation, we denote the reduced energy functional considered as mapping acting on the coefficient vector by $\bI : \R^{N_h} \to \R$.

%

\subsection{Numerical solution of the local minimization problems}\label{sec:mini}
With all the notations above, particularly the description of $\RR_h$ in \eqref{eq:equiv_Rh}, the stationary equation \eqref{eq:alg.zupdate} is equivalent to the following problem for the coefficient vector $\bz^k$: 
\begin{equation}\label{eq:reformAlg1}
  \exists \bq \in \partial I_K (\bz^k-\bz^{k-1}), \; \bp \in \partial I_\tau(\bz^k-\bz^{k-1}):  \quad D_z\bI(t_{k-1},\bz^k) + \bm + \bq + \bp = 0 .
\end{equation}
Here and for the rest of this section, we abbreviate $t_{k-1}^\ind$ simply by $t_{k-1}$ as well as $\omega = \tfrac{1}{\abs{\Omega}}$. Inserting the characterizations of $\partial I_K$ and $ \partial I_\tau$ and taking $G(\bz) = \frac{1}{2} ((\bz-\bz^{k-1})^\top \omega M (\bz-\bz^{k-1}) - \tau^2)$, we therefore find that
\begin{equation}\label{eq:minfin}
\eqref{eq:reformAlg1} \quad \Leftrightarrow \quad \left\{\quad 
 \begin{gathered}
	D_z\bI(t_{k-1},\bz^k) + \kappa \bm + \bq + \lambda_k \, \omega \, M \, (\bz^k - \bz^{k-1}) = 0, \\
	\lambda \geq 0, \quad \lambda \, G(\bz^k) = 0, \quad G(\bz^k) \leq 0 \\
	\bq \leq 0, \quad \bq^\top (\bz_k-\bz_{k-1}) = 0, \quad \bz^k - \bz^{k-1} \geq 0
 \end{gathered} \quad \right\}
\end{equation}
which can be equivalently formulated as 
\begin{equation}\label{eq:stationary_reformulation}
 \left\{\quad 
 \begin{aligned}
	D_z\bI(t_{k-1},\bz^k) + \kappa \bm + \bq + \lambda_k \, \omega \, M \, (\bz^k - \bz^{k-1}) &= 0, \\
	\max\{ -\lambda, G(\bz^k) \} &= 0, \\
	\max\{ q_i , - (z_i^k - z_{i}^{k-1}) \} &= 0 .
 \end{aligned} \quad \right\}
\end{equation}

\subsection{Numerical results}
For our numerical tests, we use two benchmark tests from \cite{hackldimitrijevic}. In any of these cases we set the external volume and surface forces to zero, so that $\ell \equiv 0$, and the softening function $g$ to $ g(z) = \exp(-z) + \varepsilon$ with $\varepsilon = 0.01$. The numerical computations are performed with \textsc{Matlab}$^\copyright$ and the linear systems of equations 
arising in each semi-smooth Newton step (cf. Section~\ref{sec:app_num}) are solved by \textsc{Matlab}'s inbuilt direct solver based on UMFPACK. The implementation for the elasticity part of the energy relies on the \textsc{Matlab} code from \cite{matlab_elasticity}.

\subsubsection*{Example I: Pre-cracked brick}

\begin{figure}[t]
\begin{minipage}{0.6\textwidth}
\center
\begin{tikzpicture}[scale=0.66]

\node at (5,4.3) {};

\fill[blue!30!white] (5,2) rectangle (10,0);
\node at (7.5,1) {$\Omega$};

\draw[thick] (0,0) rectangle (10,4); 
\draw[line width=1mm] (5,0) -- (5,0.8);
\draw[line width=1mm] (5,4) -- (5,3.2);
\draw[color=blue, thick, dashed] (5,2) -- (5,0.8); 
\draw[color=red, thick, dashed] (10,2) -- (10,0); 
\node at (5.3,1.4) {$ \Gamma_1 $};
\node at (9.7,1) {$ \Gamma_D $};


\draw[color=green!70!black, thick, dashed] (5,0.8) -- (5,0); 
\draw[color=green!70!black, thick, dashed] (5,0) -- (10,0); 
\draw[color=blue, thick, dashed] (5,2) -- (10,2);
\node at (5.4,0.4) {$ \Gamma_N $};
\node at (7.5,0.3) {$ \Gamma_N $};
\node at (7.5,1.7) {$ \Gamma_2 $};

\draw[|<->|, color=gray] (4.3,0) -- (4.3,2);
\draw[|<->| ,color=gray] (10,-0.2) -- (5,-0.2);
\draw[|<->|, color=gray] (4.8,0) -- (4.8,0.8);

\node[color=gray] at (4.6,0.4) {$ c $};
\node[color=gray] at (7.5,-0.4) {$ a $};
\node[color=gray] at (4.1,1) {$ b $};

\foreach \x in {0,1,2,3,4,5,6,7,8,9,10} {
\draw[->] (10,0.4*\x) -- (10.4,0.4*\x);
\draw[->] (0,0.4*\x) -- (-0.4,0.4*\x);
}  

\end{tikzpicture}

\end{minipage}
\hfill
\begin{minipage}{0.36\textwidth}
\center
\renewcommand{\arraystretch}{1.2}
\begin{center}
\begin{tabular}{ l r }
 \multicolumn{2}{c}{parameters} \\
 \hline
 $\kappa$ {\scriptsize{[MPa] }} & 0.1 \\ 
 $E$ {\scriptsize{[GPa] (Young's modulus)}} & 18.0 \\  
 $ \nu$ ({\scriptsize{Poisson's ratio}}) & 0.2  \\
 $\alpha$ {\scriptsize{[MPa$\cdot mm^2$]} } & 1.0 \\
 $a$ {\scriptsize{[$mm$]} } & 100 \\
 $b$ {\scriptsize{[$mm$]} } & 40 \\
 $c$ {\scriptsize{[$mm$]} } & 16
 
\end{tabular}
\end{center}
\end{minipage}

\caption{Geometry of the domain (left); Table of parameters (right).}

\label{fig:geometry_precracked}
\end{figure}
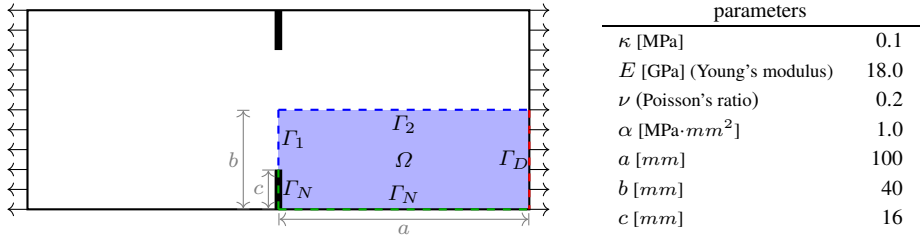

\begin{figure}[t]
   \begin{subfigure}[t]{0.45\textwidth} 
      \includegraphics[width=\textwidth, height=3cm]{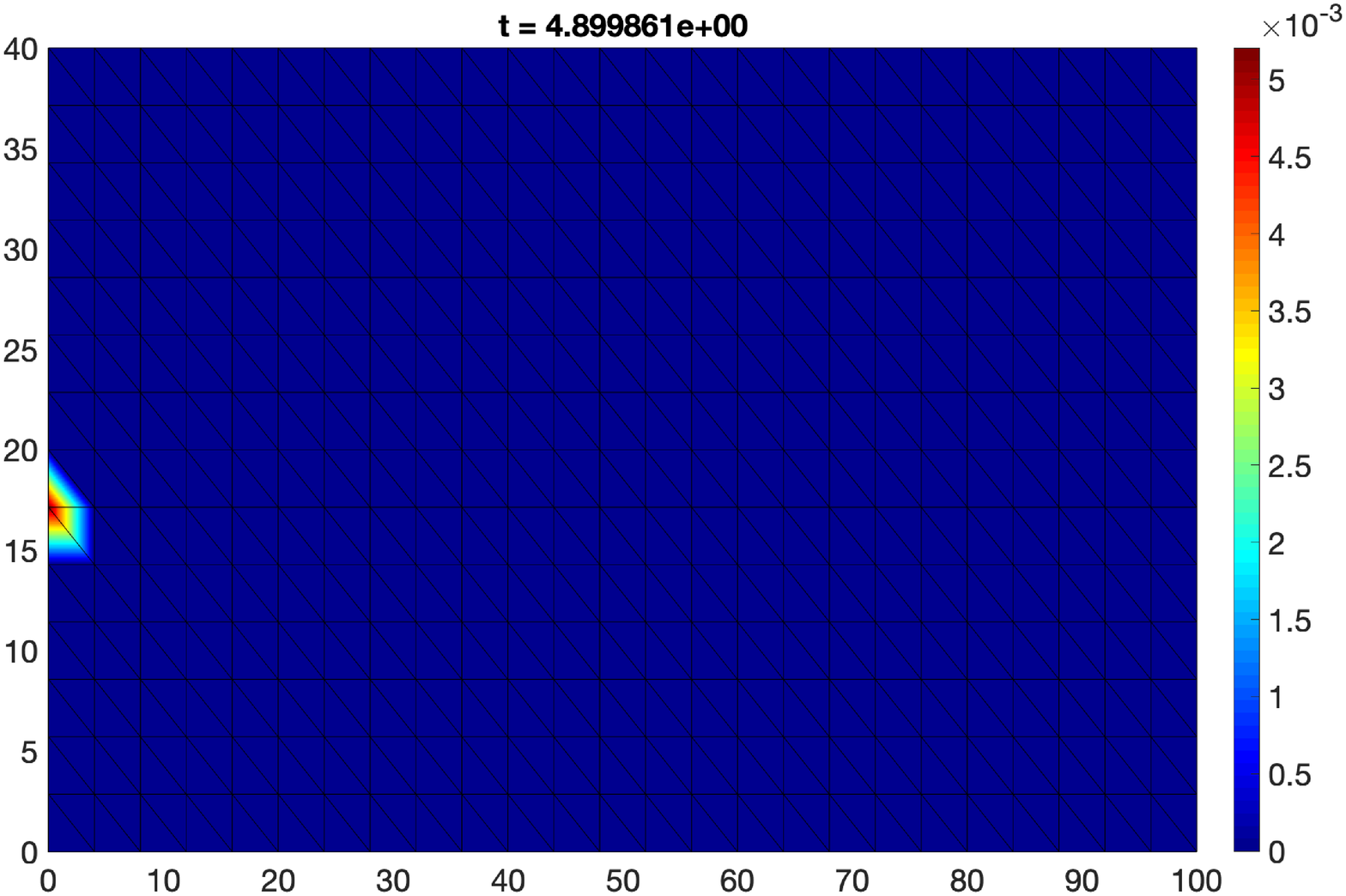} 
      \caption{$t \approx 4.9$} 
      \label{fig:fig1} 
   \end{subfigure}\hfill%
   \begin{subfigure}[t]{0.45\textwidth} 
   \includegraphics[width=\textwidth, height=3cm]{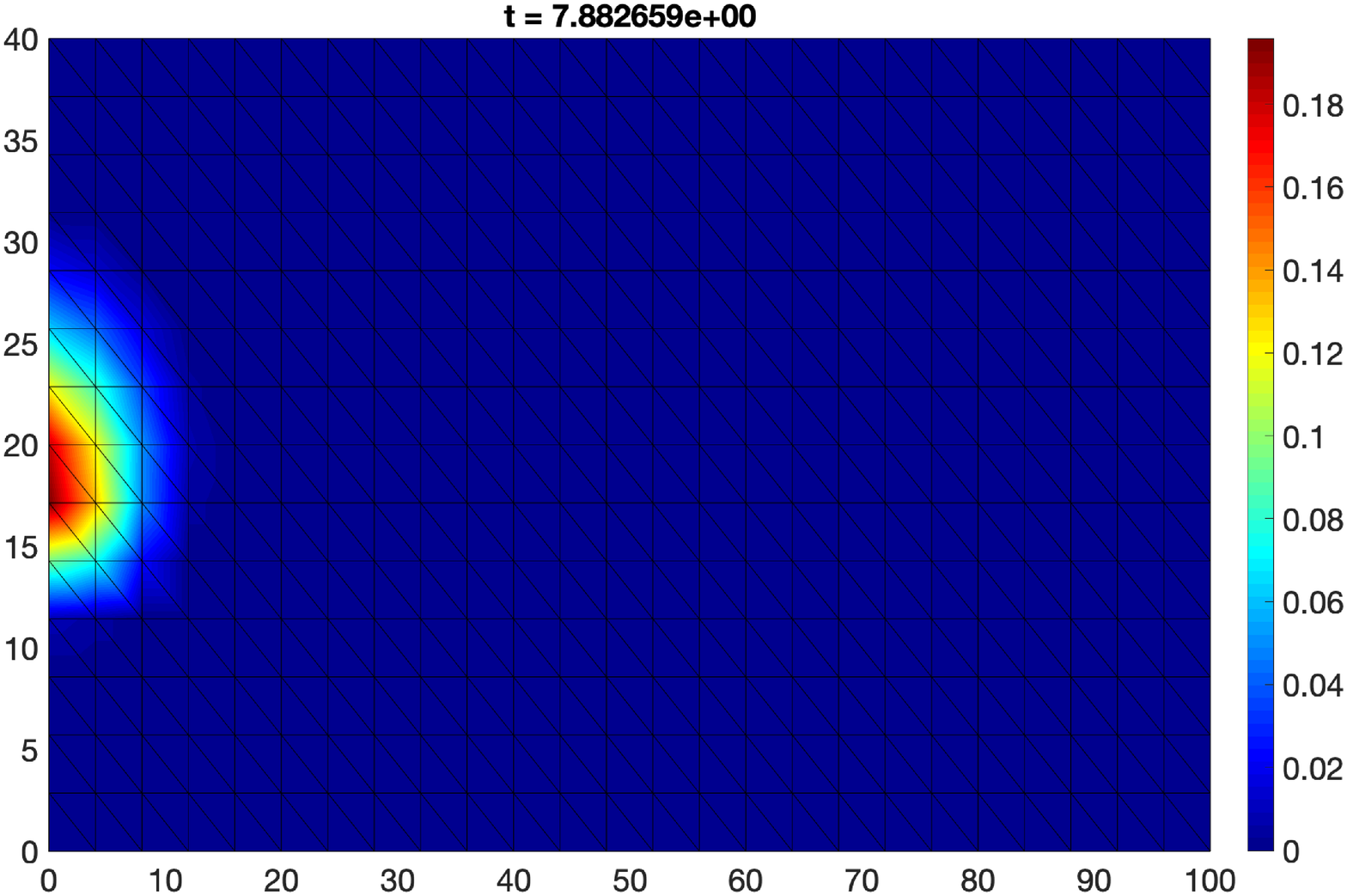}
      \caption{$ t \approx 7.88 $} 
      \label{fig:fig2} 
   \end{subfigure}
   \begin{subfigure}[t]{0.45\textwidth} 
      \includegraphics[width=\textwidth, height=3cm]{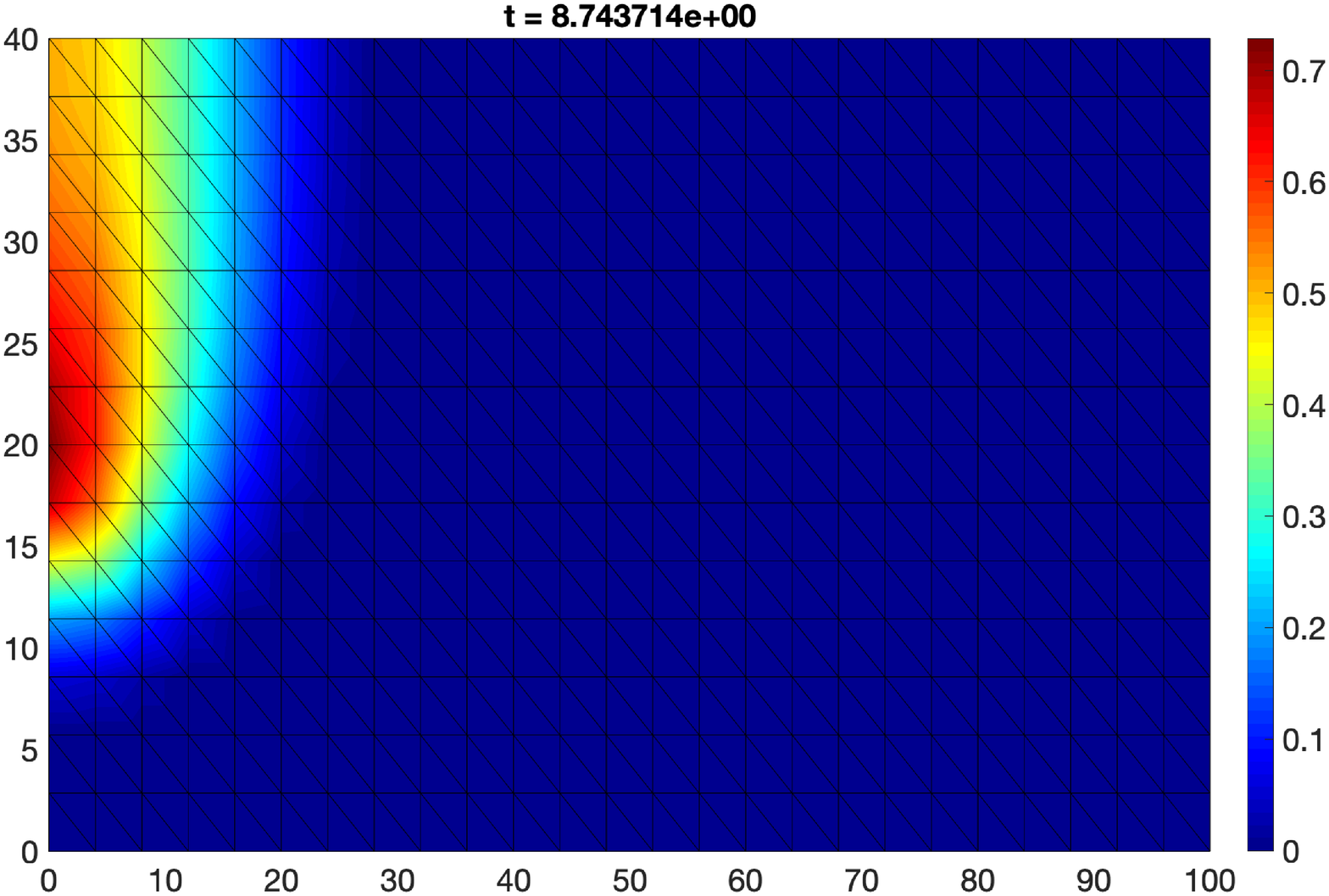} 
      \caption{$ t \approx 8.74$ } 
      \label{fig:fig3} 
   \end{subfigure}\hfill%
   \begin{subfigure}[t]{0.45\textwidth} 
   \includegraphics[width=\textwidth, height=3cm]{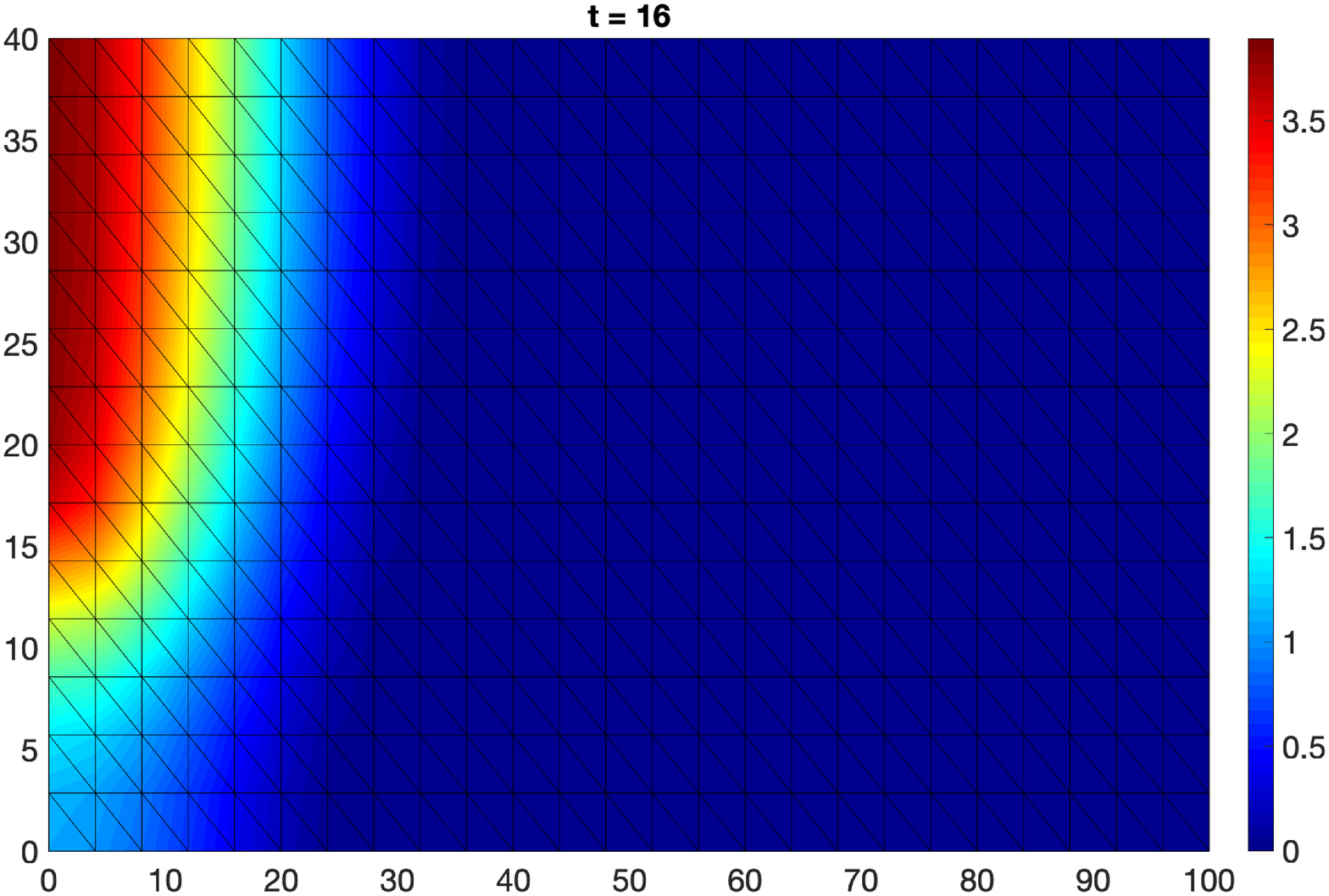}
      \caption{$ t = 16$ } 
      \label{fig:fig4} 
   \end{subfigure}
   \caption{State of the dimensionless damage variable $z_h$ at different points in time. Note that the distribution of the colormap varies with the time.}
   
  \label{fig:results_precracked}
\end{figure}

The geometry for this example is shown in Figure~\ref{fig:geometry_precracked}. Due to its symmetry the computation is performed using only one quarter of the whole system. Therefore, the symmetry axes become parts of the boundary of $\Omega$ and we impose the following symmetry boundary conditions
\[ u_1 = 0 \quad \text{on } \Gamma_1 = \{ 0 \} \times [16, 40] \quad \text{ and }\quad u_2 = 0 \quad \text{on } \Gamma_2 = [0,100] \times \{ 40 \} . \]
During the time interval $[0,16]$ the workpiece is strechted at its ends, which we realize by Dirichlet conditions on $\Gamma_D$, i.e. \[ u_1 = t , \quad u_2 = 0 \quad \text{on } \Gamma_D = \{ 100 \} \times [0, 40] . \] 
Furthermore, we use the parameters given in the table of Figure~\ref{fig:geometry_precracked} (right). We initiate the evolution with $ z_0 \equiv 0$ and choose $\tau = 0.1$ for the time step size. The state of damage at several points during the time interval $[0,16]$ are shown in Figure~\ref{fig:results_precracked}. Obviously, damage occurs at first at the tip of the crack and evolves along the symmetry axis $ \Gamma_1$ afterwards. 
As expected, it also concentrates on regions with large stresses. However, the sharpness of interfaces between damaged and undamaged areas highly depends on the choice of the functional $\II_1$, see \cite{bartels2018numerical}. In our case, this interface is rather diffuse and cannot be sharpened by a refinement of the mesh, see Figure~\ref{fig:precracked_refined}. Clearly, one may reduce the factor $\alpha$ included in the operator $A$. However, this leads to instabilities in the semismooth Newton method and globalization strategies might be necessary, which is subject of future research.
Nevertheless, the results are stable with respect to mesh refinement, which can be observed in Figure~\ref{fig:force_displacement_diagram} that shows the force-displacement-diagram for three different mesh sizes. The reaction force is calculated by integrating the stress $\sigma$ in normal direction along the boundary $\Gamma_D$. In comparison with the results from \cite{hackldimitrijevic}, there is a high degree of conformity, except for the larger values of the reaction force here after reaching its maximum at approximately $ 0.08 \, mm $ of displacement, cf. also with \textsl{Example II}. With regard to the time function $\hat{t}$ that is depicted in Figure~\ref{fig:time} we see that the spreading of the damage area starting at $ t \approx 8 $ is slightly faster than the rest of the evolution but does not cause a jump. 
\begin{figure}[t]
   \centering
      \includegraphics[width=0.8\textwidth]{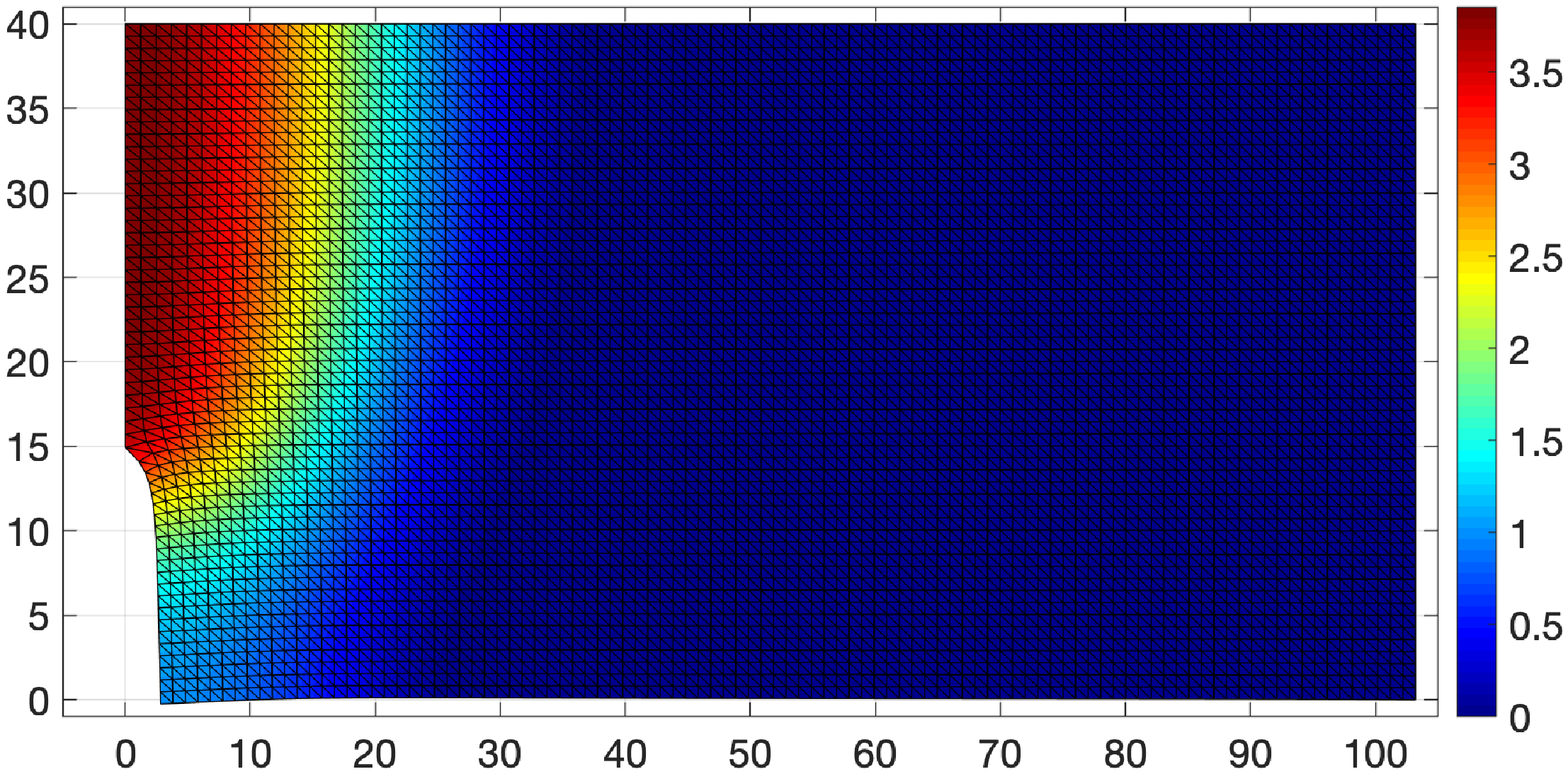} 
      \captionof{figure}{Refined and 
      deformed mesh combined with the state of the dimensionless damage variable at the final time $ t = 16$. Note that the displacement is magnified by a factor of $20$.} 
      \label{fig:precracked_refined} 
      
   \begin{minipage}[t]{0.48\textwidth} 
   \centering
   \includegraphics[width=0.85\textwidth]{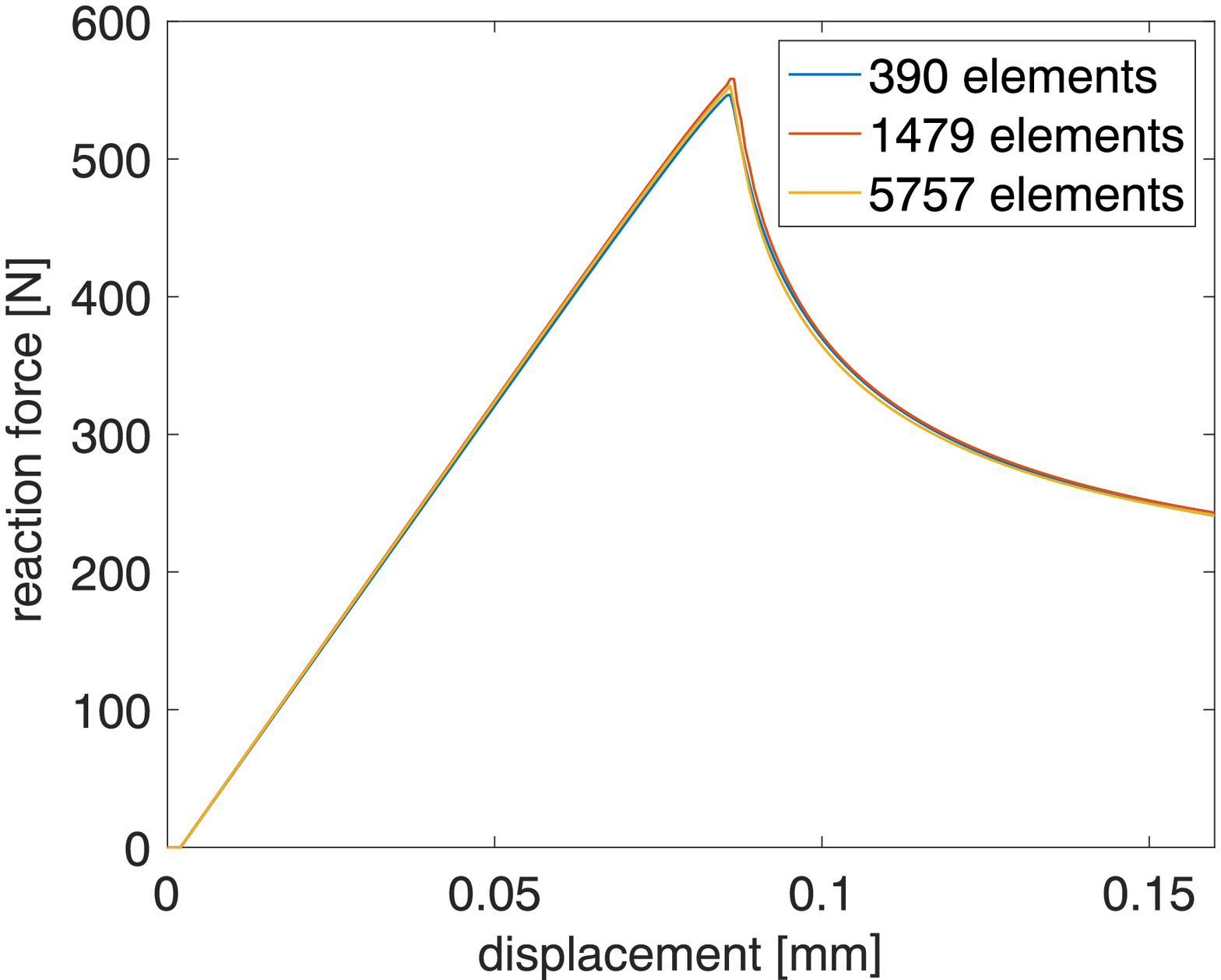}
      \captionof{figure}{Reaction Force on the boundary $\Gamma_D$ depending on the displacement $u_D$ for three different mesh sizes.} 
      \label{fig:force_displacement_diagram} 
   \end{minipage}\hfill
      \begin{minipage}[t]{0.48\textwidth} 
   \centering
   \includegraphics[width=0.85\textwidth]{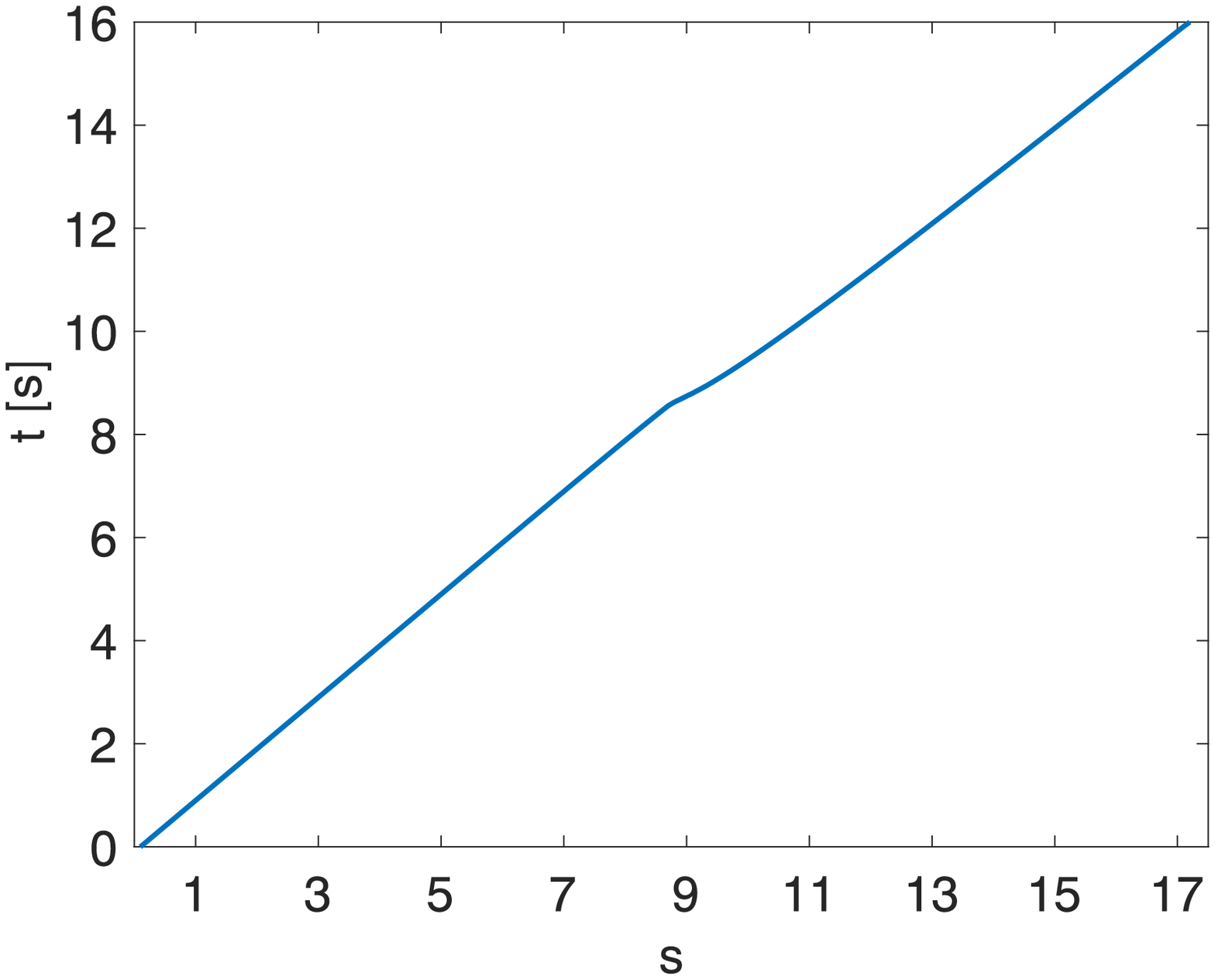}
      \captionof{figure}{Function $ \hat{t} $ in dependence of the artificial time $s$.} 
      \label{fig:time} 
   \end{minipage}
\end{figure}

Let us finally compare the results obtained with an alternative approach. Instead of incorporating the constraint $ \norm{z_k-z_{k-1}}_\VV \leq \tau $ which corresponds to an $L^2$-ball of size $ \tau$ it is convenient to take $ \norm{z_k-z_{k-1}}_{L^\infty} \leq \tau $. Indeed, as we have already imposed lower bounds for $ \bz^k$, namely $ z^k_i \geq z^{k-1}_i $, it seems appropriate from a numerical point of view to also include $ z^k_i \leq z^{k-1}_i + \tau $ as a constraint since this leads to pointwise box constraints that are particularly well suited to semismooth Newton-methods, see e.g. \cite{HIKprimaldual02}. Unfortunately, this choice is not covered by our convergence analysis for the local stationarity scheme (\nameref{alg:locmin}). Nevertheless, we provide the numerical results that can be obtained by using this version. Thus, while we change \eqref{eq:stationary_reformulation} to
\begin{equation}
 \left\{\quad 
 \begin{aligned}
	D_z\bI(t_{k-1},\bz^k) + \kappa \bm + \bq  &= 0, \\
	\max\{ q_i, (z_i^k - z_{i}^{k-1}) - \tau \} &= 0 \\
	\max\{ q_i , - (z_i^k - z_{i}^{k-1}) \} &= 0 
 \end{aligned} \quad \right\}
\end{equation}
we keep all parameters unaltered. Clearly, the Newton-matrix $H_n$ also changes in the obvious way. The difference between both solutions using \nameref{alg:locmin} is shown in Figure~\ref{fig:damage_difference}. The time-function $\hat{t}$ for the solution using box-constraints is given in Figure~\ref{fig:time_hole}. It is easy to see that both solutions, using $L^2$- and $L^\infty$-constraints, are very similar. Indeed, the $L^2$ and $L^\infty$ distance at the end time $ T = 16 $ calculates to $ \norm{z_{h,2} - z_{h,\infty}}_{L^2(\Omega)} = 2.5 \cdot 10^{-2}$. Since the incorporation of box constraints is natural in this context and its implementation is also easier to realize it provides an interesting topic for further research.

A possible approach for this is to consider the $ p$-Laplacian with $ p>d$ for the regularization instead of $ A = - \laplace$ as proposed in \cite{krz15}. In this case, we have $ \ZZ = W^{1,p}(\Omega) \embeds^{c,d} C(\overline{\Omega})$ which may open up the opportunity to use the $L^\infty$-norm as constraint in \eqref{eq:locmin_h}. 

\begin{figure}[t]
   \begin{minipage}[t]{0.48\textwidth} 
      \includegraphics[width=\textwidth]{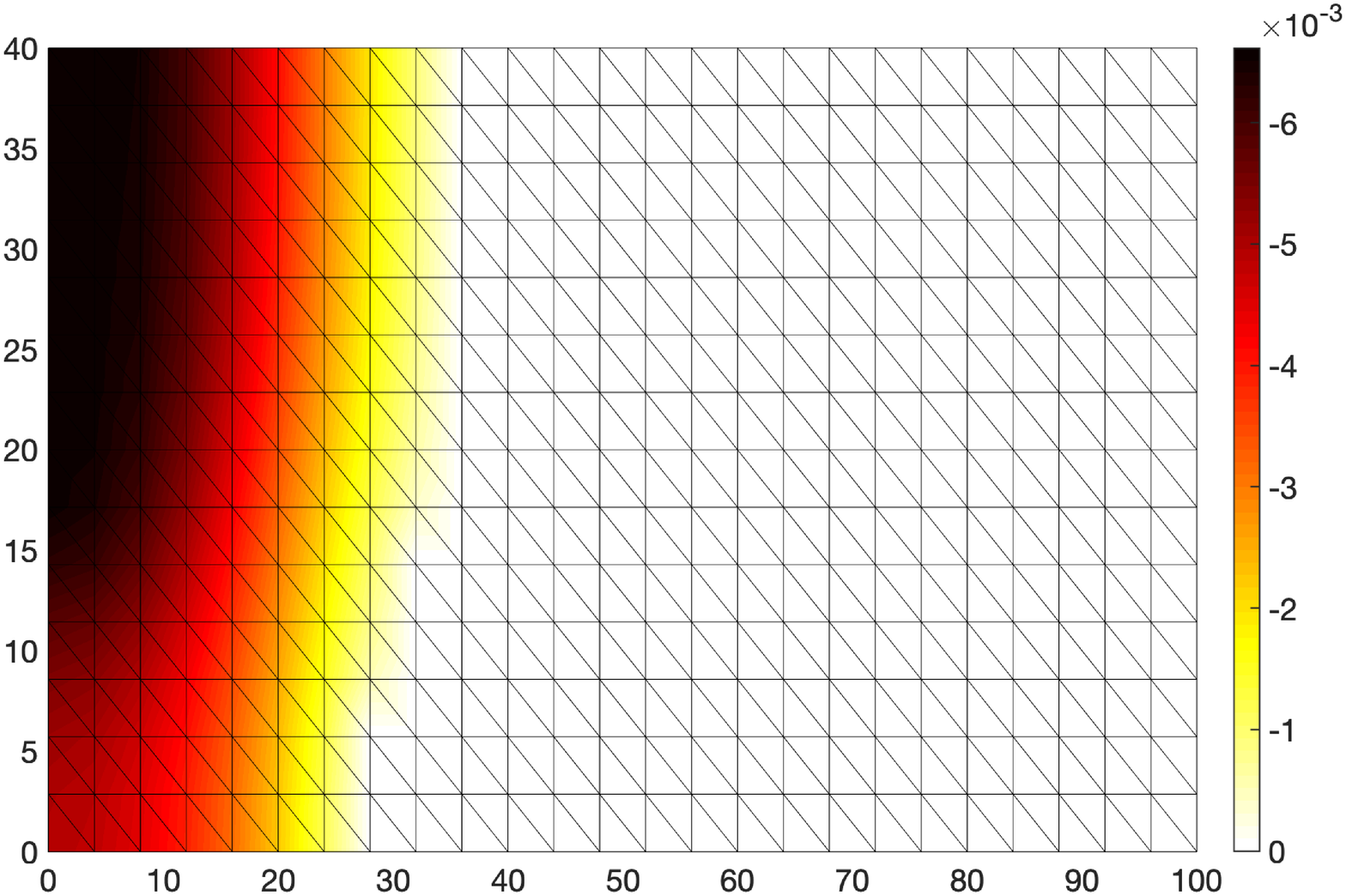} 
      \captionof{figure}{Difference $ (z_{h,2} - z_{h,\infty})$ of the two solutions obtained by using the $L^2$- and box-constraints at the endtime.} 
      \label{fig:damage_difference} 
   \end{minipage}\hfill%
   \begin{minipage}[t]{0.48\textwidth} 
   \includegraphics[width=0.85\textwidth]{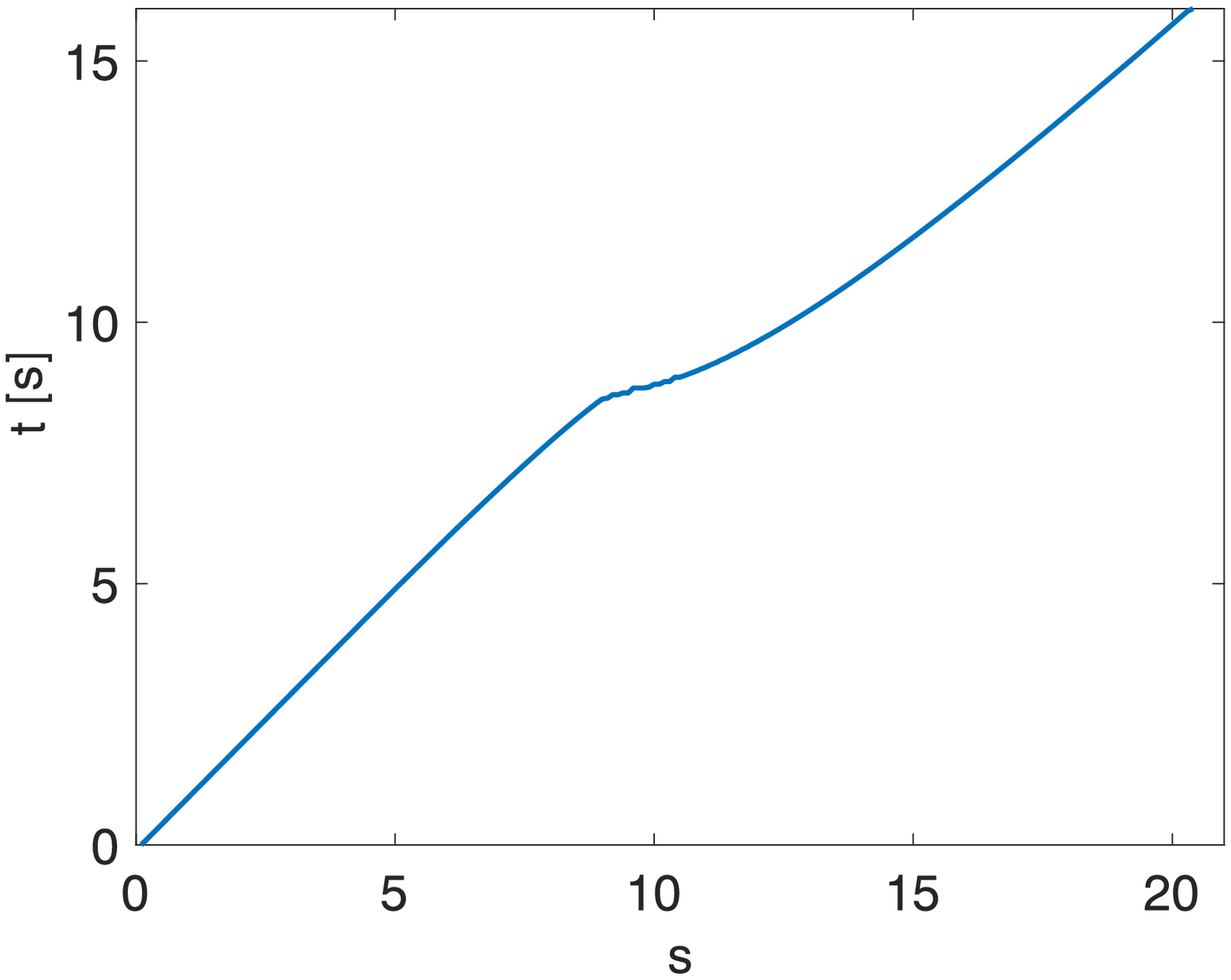}
      \captionof{figure}{Function $ \hat{t} $ in dependence of the artificial time $s$ for the solution obtained by using box-constraints.} 
      \label{fig:time_hole} 
   \end{minipage}
\end{figure}



\subsubsection*{Example II: Brick with a circular hole}
As proof-of-concept we consider the situation depicted in Figure~\ref{fig:geometry_circular}. Again, due to its symmetry the computation is performed using only one quarter of the whole system. This also implies certain symmetry conditions for the boundary of the domain $\Omega$. For the situation at hand, we impose 
\[ u_1 = 0 \quad \text{on } \Gamma_1 = \{ 0 \} \times [50, 100] \quad \text{ and }\quad u_2 = 0 \quad \text{on } \Gamma_2 = [50,100] \times \{ 0 \} . \]
Moreover, the workpiece is pulled apart at two opposite sides, which we realize by the Dirichlet conditions \[ u_1 = 0 , \quad u_2 = t \quad \text{on } \Gamma_D = [0,100] \times \{ 100 \} . \] 
The parameters used are given in the table of Figure~\ref{fig:geometry_circular} (right). Finally, we initiate the evolution with $ z_0 \equiv 0$ and choose $\tau = 0.1$ for the time step size.
The numerical solution of the damage variable for an intermediate time point $ t \approx 7.8 $ is shown in Figure~\ref{fig:damage_hole} and the corresponding force-displacement-diagram is depicted in Figure~\ref{fig:force_hole}. We observe a strong similarity with the results in \cite{hackldimitrijevic}. However, the reaction force in the end phase of the evolution is significantly higher in our case. This may be a result of the additional regularization by the introduction of a "local field" $d$ in \cite{hackldimitrijevic}. Certainly, this requires further investigation.

\begin{figure}
\begin{minipage}{0.6\textwidth}
\center
\begin{tikzpicture}[scale=0.55]

\def\innerCircle{(5,5) circle (2cm)}
\def\omega{(5,5) rectangle (9,9)}
\def\omegaHelp{(5.03,5.03) rectangle (9,9)}
\def\workpiece{(1,1) rectangle (9,9)}

\begin{scope}
\clip \omegaHelp;
\fill[blue!30!white, even odd rule] \omega \innerCircle;
\end{scope}

\draw[thick] \workpiece; 
\draw[thick] \innerCircle;

\begin{scope}
\clip \omegaHelp;
\draw[thick, color=green!70!black, dashed] \innerCircle;
\end{scope}

\draw[|<->|, color=gray] (5,9.6) -- (9,9.6);
\draw[|<->|, color=gray] (9.4,5) -- (9.4,9);
\draw[|<->|, color=gray] (5+1.4,5+1.4) -- (5,5);
\node[color=gray] at (5.7,5.9) {$ r $};
\node[color=gray] at (9.6,7) {$ a $};
\node[color=gray] at (7,9.8) {$ a $};

\draw[color=blue, thick, dashed] (7,5) -- (9,5); 
\draw[color=green!70!black, thick, dashed] (9,5) -- (9,9); 
\draw[color=red, thick, dashed] (9,9) -- (5,9);
\draw[color=blue, thick, dashed] (5,9) -- (5,7);
\node at (8.6,7) {$ \Gamma_N $};
\node at (7,8.7) {$ \Gamma_D $};
\node at (4.7,8) {$ \Gamma_1 $};
\node at (6.7,6.7) {$\Gamma_N$};
\node at (7.7,4.7) {$\Gamma_2$};

\foreach \x in {0,1,2,3,4,5,6,7,8,9,10,11,12,13,14,15,16,17,18,19,20} {
\draw[->] (1+0.4*\x,1) -- (1+0.4*\x,.6);
\draw[->] (1+0.4*\x,9) -- (1+0.4*\x,9.4);
}  

\node at (7.4,7.2) {$\Omega$};

\end{tikzpicture}

\end{minipage}
\hfill
\begin{minipage}{0.39\textwidth}
\center
\renewcommand{\arraystretch}{1.5}
\begin{center}
\begin{tabular}{ l r }
 \multicolumn{2}{c}{parameters} \\
 \hline
 $\kappa$ {\scriptsize{[MPa]} } & 0.1 \\ 
 $E$ {\scriptsize{[GPa] (Young's modulus)}} & 18.0 \\  
 $ \nu$ ({\scriptsize{Poisson's ratio}}) & 0.2  \\
 $\alpha$ {\scriptsize{[MPa$\cdot mm^2$]} } & 1.0 \\
 $a$ {\scriptsize{[$mm$]} } & 100.0 \\ 
 $r$ {\scriptsize{[$mm$]} } & 50.0 
 
\end{tabular}
\end{center}
\end{minipage}

\caption{Geometry of the domain (left); Table of parameters (right).}

\label{fig:geometry_circular}
\end{figure}

\begin{figure}
   \begin{minipage}[t]{0.48\textwidth} 
   \centering
      \includegraphics[width=0.85\textwidth]{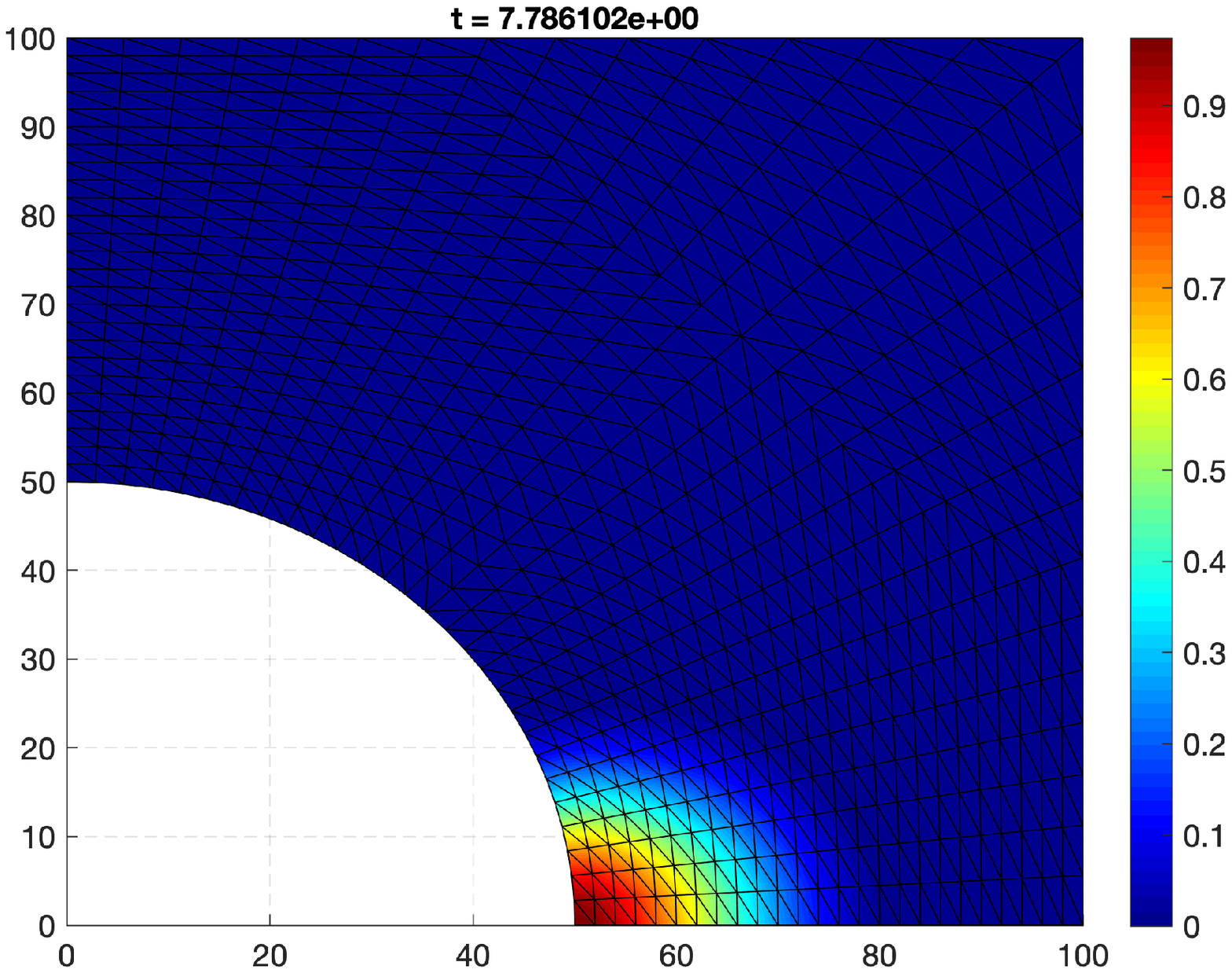} 
      \captionof{figure}{State of the dimensionless damage function $z_h$ at time point $ t \approx 7.8$} 
      \label{fig:damage_hole} 
   \end{minipage}\hfill%
   \begin{minipage}[t]{0.48\textwidth} 
   \centering
   \includegraphics[width=0.85\textwidth]{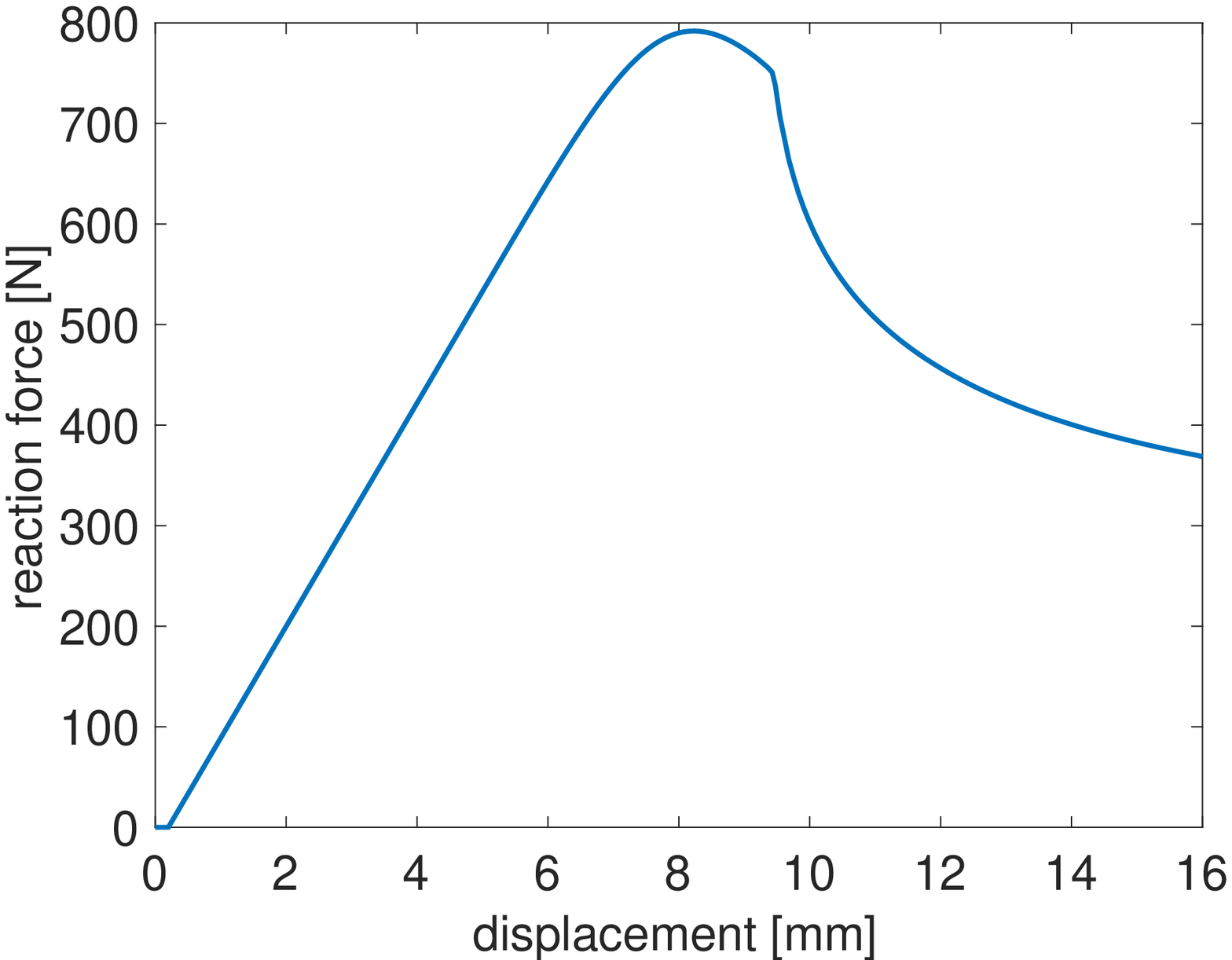}
      \captionof{figure}{Reaction force on the boundary $\Gamma_D$ depending on the displacement $u_D$.} 
      \label{fig:force_hole} 
   \end{minipage}
\end{figure}

%

\section{Conclusions}

We presented a numerical scheme for the approximation of parametrized solutions for rate-independent systems including a fairly general setting for the energy and dissipation functionals.
The scheme itself is based on the local minimization scheme introduced in \cite{efenmielke06}, but relies on stationary points rather than local minima, making it very accessible for numerical optimization algorithms (limit points are, in general, stationary). Moreover, by adapting the convergence analysis of the recent contributions in \cite{knees17,fem_paramsol} and using arguments from \cite{krz13}, we proved that the scheme provides parametrized solutions of the original rate-independent system under Mosco-convergence of approximations for the dissipation $\RR$. While this is at first glance a result that verifies the consistency of the local incremental stationarity scheme, we, moreover, gain an existence result for parametrized solutions in the case of a nonconvex energy and unbounded dissipation. We then focused on the realization of our scheme for a model of the evolution of damage within a workpiece. We employed a finite element discretization in space and used a semi-smooth Newton method for solving the discrete stationary system arising in each step of the scheme. The resulting algorithm behaves efficient and robust in our numerical tests. 
Afterall there are several topics for future research. This concerns for example the usage of an $L^\infty$-norm in the indicator functional in \eqref{eq:alg.zupdate} which leads to an easy to implement algorithm for the damage model considered in this paper. In the same context, it might be interesting to relax the assumptions for the energy in order to incorporate functionals that allow for a sharper resolution of the interface between damaged and undamaged regions. Moreover, considering dissipation functionals which are also depending on the state $z$, i.e., $\RR = \RR(z,z^\prime)$ should be noted here. As there are only few results in this direction for rate-independent systems this does not only concern parametrized solutions.

\begin{acknowledgement}
I would like to thank Christian Meyer (TU Dortmund) for various discussions on the topic.
\end{acknowledgement}

\begin{appendix}

\section{Auxiliary results from convex analysis}

In this section, we collect some useful properties of $\RR_\delta$ and $I_\tau$, respectively. Since most of the results are quite standard, we keep the arguments brief.

\begin{Lemma}\label{lem:char.subdiff}
Let $\WW$ be a normed vector space and $\JJ: \WW \to \R$ a convex and positive 1-homogeneous functional. Then it holds
\begin{subequations}
\begin{gather}
\partial \JJ(v) \subset \partial \JJ(0) \quad \forall v \in \WW \label{eq:app.subdiff0}\\
\xi \in \partial \JJ(0) \; \Longleftrightarrow \; \JJ(w)  \geq \dual{\xi}{w} \quad \forall w \in \WW \label{eq:app.subdiff1}\\
\partial \JJ(v) = \{ \xi \in \partial \JJ(0) \, : \, \JJ(v) = \dual{\xi}{v} \} \label{eq:app.subdiff3}\\
\JJ^\ast(\xi) = I_{\partial \JJ(0)}(\xi) \quad \forall \xi \in \WW^*\label{eq:app.conjfun}
\end{gather}
\end{subequations}
where $I_{\partial \JJ(0)}$ denotes the indicator functional of $\partial \JJ(0)$. 
\end{Lemma}
Let us define the indicator functional $ I_\tau : \VV \to \R \cup \{ \infty \}$ as 
\begin{equation}\label{eq:inditau}
I_\tau(v) := \begin{cases} 0, & \text{if } \norm{v}_\VV^2 \leq \tau^2 \\ +\infty , & \text{else}. \end{cases}
\end{equation}
As in the proof of Lemma~\ref{prop.optimalityprops}, we abbreviate 
$\RR_{\tau,\delta} = \RR_\delta + I_\tau$.

\begin{Lemma}\label{lem:conv.ana}
For every $\eta \in \ZZ^*$, there holds
\begin{equation}
 (\RR_{\tau,\delta})^\ast(\eta) = \tau \dist_{\VV^*}\{\eta, \partial \RR_\delta(0)\}, \label{eq:eq.A2}
\end{equation}
where $\dist_{\VV^*}\{\eta,\partial\RR_\delta(0)\} = \inf\{\norm{\eta-w}_{\VV^*} \, : \, w \in \partial \RR_\delta(0)\}$. Particularly, $ \dist_{\VV^*}\{\eta,\partial\RR_\delta(0)\} = \infty $ if there exists no $ w \in \partial \RR_\delta(0) $ such that $ \eta - w \in \VV^*$. Moreover, if $\dist_{\VV^*}\{\eta,\partial\RR_\delta(0)\} < \infty$ then there exists $ w \in \partial\RR_\delta(0) $ such that $\dist_{\VV^*}\{\eta,\partial\RR_\delta(0)\} = \norm{\eta-w}_{\VV^*} $.
\end{Lemma}

\begin{proof}
We use the inf-convolution formula (see \citep[Prop.~3.4]{attouch}), 
which is applicable, since both functions are proper, convex and closed and we have $ \operatorname{dom}(I_\tau) = B_\VV(0,\tau)$. This gives
\begin{equation}\label{eq:infconv}
 \left(\RR_\delta + I_\tau\right)^\ast(\eta) = 
 \inf_{w\in \VV^*}\left(\RR_\delta^\ast(w)+ I_\tau^\ast(\eta-w)\right).
\end{equation}
For $I_\tau^*$, direct calculation leads to
\begin{equation}\label{eq:Itaufenchel}
 I_\tau^{\ast}(\eta) 
 = \begin{cases} \tau \norm{\eta}_{\VV^*}, & \text{if } \eta \in \VV^* \\ +\infty, & \text{if } \eta \in \ZZ^* \setminus \VV^*. \end{cases}
\end{equation}
Moreover, Lemma~\ref{lem:char.subdiff} gives
$\RR_\delta^\ast(\eta) = I_{\partial\RR_\delta(0)}(\eta)$. Inserting this together with \eqref{eq:Itaufenchel} in \eqref{eq:infconv} finally yields
\begin{align*}
(\RR_\delta + I_\tau)^\ast(\eta)
= \inf_{w \in \partial\RR_\delta(0)} \{\tau \norm{\eta-w}_{\VV^*} \}
= \tau \dist_{\VV^*}( \eta, \partial\RR_\delta(0)),
\end{align*}
which is \eqref{eq:eq.A2}. Now, let $ \dist_{\VV^*}( \eta, \partial\RR_\delta(0)) < \infty $ and take $ w_k \in \partial\RR_\delta \subset \ZZ^*$ such that $\lim_{k \to \infty} \norm{\eta - w_k}_\VV^* = \dist_{\VV^*}( \eta, \partial\RR_\delta(0))$. Obviously, this implies that $ \mu_k := \eta - w_k$ is uniformly bounded in $\VV^*$. Hence, we can extract a weakly convergent subsequence (w.l.o.g. denoted by the same symbol) such that $ \mu_k \weakly \mu $ in $ \VV^*$ for some $ \mu \in \VV^*$.  Therefore, by the embedding $\VV^* \embeds \ZZ^*$, we have $ w_k = \eta - \mu_k \weakly \eta - \mu =: w$ in $\ZZ^*$ which implies that $\mu = \eta - w$. To proceed, we note that $\partial\RR_\delta(0) \subset \ZZ^*$ is again convex and closed and thus weakly closed in $\ZZ^*$. Since $ w_k \in \partial\RR_\delta(0)$ for all $ k$ we also have $ w \in \partial\RR_\delta(0)$. Finally, the weak lower semicontinuity of the norm gives
\[ \dist_{\VV^*}( \eta, \partial\RR_\delta(0)) \leq \norm{\eta - w}_{\VV^*} = \norm{\mu}_{\VV^*} \leq \liminf_{k \to \infty} \norm{\mu_k}_{\VV^*} = \dist_{\VV^*}( \eta, \partial\RR_\delta(0)) \]
which finishes the proof.
\end{proof}

In view of \eqref{eq:alg.zupdate} we note that $ \partial^\ZZ I_\tau(z) = \partial^\VV I_\tau(z) $ for any $ z \in \ZZ$ which can be easily obtained by considering the projection operator $ \Pi_\VV : \ZZ \to \VV$ and applying the chain-rule for subdifferentials to $ I_\tau \circ \Pi_\VV$. As in Remark~\ref{rem:subdiffRR_weakly_closed} we nevertheless simply write $\partial I_\tau(z)$. Moreover, we have the following characterization.
\begin{Lemma}\label{lem:subdiffItau}
 Let $\VV$ be a reflexive Banach space and $v \in \VV$ be arbitrary. Then, $\xi \in \VV^*$ is an element of $\partial I_\tau(v)$, iff
 \begin{equation}\label{eq:subdiffItau_char}
  \norm{v}_\VV \leq \tau, \quad \norm{\xi}_{\VV^*} (\norm{v}_\VV - \tau) = 0, \quad \dual{\xi}{v} \geq \norm{\xi}_{\VV^*} \tau.
 \end{equation}
If $ \VV $ is even a Hilbert space, then $ \xi \in \partial I_\tau(v)$ iff there exists a multiplier $\lambda \in \R$ such that $ \xi = \lambda \, J_\VV v $ and
 \begin{equation}\label{eq:subdiffItau_hilbert}
  \norm{v}_\VV \leq \tau, \quad \lambda (\norm{v}_\VV - \tau) = 0, \quad \lambda \geq 0.
 \end{equation}
\end{Lemma}

\begin{proof}
According to a classical result of convex analysis in combination with \eqref{eq:Itaufenchel}, it holds
\begin{align}
 \xi \in \partial I_\tau(v) \quad \Longleftrightarrow \quad I_\tau(v)+I_\tau^*(\xi) = \dual{\xi}{v} 
 \quad \Longleftrightarrow \quad
 \left\{\;\begin{aligned}
  & \norm{v}_\VV \leq \tau \\
  &\tau \norm{\xi}_{\VV^*} = \dual{\xi}{v}
 \end{aligned}\right. 
\label{eq:eq.auxAAA}
\end{align}
Hence, \eqref{eq:subdiffItau_char} follows easily from $ \tau \norm{\xi}_{\VV^*} = \dual{\xi}{v} \leq \norm{v}_\VV \norm{\xi}_{\VV^*}$. Now, assume that $\VV$ is a Hilbert space. Then, the equality in \eqref{eq:eq.auxAAA} can only hold if $\xi = \lambda \, J_\VV v$ for some $\lambda \in \R$. 
Inserting this into \eqref{eq:eq.auxAAA}, we conclude that $\lambda \geq 0$ and so, if $\norm{v}_\V < \tau$, then $\xi = 0$ which gives \eqref{eq:subdiffItau_hilbert}.
\end{proof}	

\section{Numerical aspects for the discrete energy functional}\label{sec:app_num}

In this section we formally derive the necessary derivatives of $\bI$ used for the implementation of the local stationarity scheme in case of the damage model. We start with $ D_z\bI$ for which we note that by $ D_u \tilde{\EE}(t,z_h,S_h(z_h)) = 0$ from \eqref{eq:discr_sol_operator}, one obtains
\begin{multline}\label{eq:discr_first_derivative}
D_z \tilde{\II}(t , z_h) v_h  = D_z \tilde{\EE}(t,z_h,S_h(z_h)) v_h + D_u \tilde{\EE}(t,z_h,S_h(z_h)) [ S_h^\prime(z_h) v_h ] \\
= D_z \tilde{\EE}(t,z_h,S_h(z_h)) v_h \quad \forall v_h \in \ZZ_h . 
\end{multline}
Hence, with a little abuse of notation (particularly identifying $ \bz $ and $z_h$) we have 
\[ \Big( D_z \bI(t , \bz) \Big)_{i=1}^{N_h} = \Big( \dual{D_z \tilde{\EE}(t,z_h,S_h(z_h))}{\varphi_i}_{\ZZ^*,\ZZ} \Big)_{i=1}^{N_h}. \]
The fact that $  D_u \tilde{\EE}(t,z_h,S_h(z_h)) = 0 $ moreover implies that 
\begin{align*}
0 &= D_z \{ D_u \tilde{\EE}(t,z_h,S_h(z_h))\, v_h \} \, w_h \\
&= D^2_{zu} \tilde{\EE}(t,z_h,S_h(z_h))[v_h,\, w_h] + D^2_{uu} \tilde{\EE}(t,z_h,S_h(z_h))[v_h, \, S_h^\prime(z_h)w_h] \\
\end{align*}
Therefore, for arbitrary $ w_h $, we can characterize $\eta_h := S_h^\prime(z_h) w_h$ as the solution of
\begin{equation} \label{eq:adjoint_solution_operator}
0 = D^2_{zu} \tilde{\EE}(t,z_h,S_h(z_h))[v_h, \, w_h] + D^2_{uu} \tilde{\EE}(t,z_h,S_h(z_h))[v_h, \, S_h^\prime(z_h)w_h] \quad \forall w_h \in \UU_h .
\end{equation}
Consequently, exploiting \eqref{eq:discr_first_derivative} it holds
\begin{align*}
D_z^2 \tilde{\II}(t , z_h) [v_h,w_h] &=  D_u D_z \tilde{\EE}(t,z_h,S_h(z_h))[v_h,S_h^\prime(z_h) w_h] + D_z D_z \tilde{\EE}(t,z_h,S_h(z_h))[v_h,w_h] \\
&= D^2_{uz} \tilde{\EE}(t,z_h,S_h(z_h))[v_h,\eta_h] + D^2_{zz} \tilde{\EE}(t,z_h,S_h(z_h))[v_h,w_h]
\end{align*}
with $ \eta_h = S_h^\prime(z_h) w_h $ solving \eqref{eq:adjoint_solution_operator}. Again, with a little abuse of notation we thus have 
\begin{multline}\label{eq:D2z_boldI}
\Big( D_z^2 \bI(t , \bz) \Big)_{i,j=1}^{N_h} = \Big( \dual{D^2_{uz} \tilde{\EE}(t,z_h,S_h(z_h)) \varphi_i}{S_h^\prime(z_h)\varphi_j }_{\ZZ^*,\ZZ} \\
+ \dual{D^2_{zz} \tilde{\EE}(t,z_h,S_h(z_h)) \varphi_i}{\varphi_j}_{\ZZ^*,\ZZ} \Big)_{i,j=1}^{N_h} . \end{multline}
Now, let us turn to the semismooth Newton-method that is used in order to solve the stationary system in \eqref{eq:stationary_reformulation}. In general, if we denote the left hand side of \eqref{eq:stationary_reformulation} by $F:\R^{2N_h+1} \to \R^{2N_h+1}$, equation \eqref{eq:stationary_reformulation} becomes $F(\bz, \bq, \lambda) = 0$ and we need to solve the following semi-smooth Newton equation
\begin{equation*}
 H_n \, (\bx^{n+1} - \bx^n) = - F(\bx^n) 
 \quad \text{with} \quad
 H_n \in \partial^N F(\bx^n),
\end{equation*}
with the iterate $\bx^n = (\bz^n, \bq^n, \lambda^n)$ and a Newton-derivative $\partial^N F$. However, the matrix $H_n$ contains second order information of $\bI$ which, by \eqref{eq:D2z_boldI} necessitates the determination of $ S_h^\prime(z_h)$. In order to keep track of this, we blow up the whole system so that in every semismooth Newton-step we actually solve
\begin{equation*}
 \tilde{H}_n \begin{pmatrix} \Delta \bz^n \\ \mathbf{\eta}^n \\ \Delta \bq^n \\ \Delta \lambda^n \end{pmatrix} 
= - \begin{pmatrix}
D_z\bI(t_{k-1},\bz^n) + \kappa \bm + \bq^n + \lambda^n \, \omega \, M \, (\bz^n - \bz^{k-1}) \\
0 \\
\max\{q_i^n ,-( z_i^n - z_{i}^{k-1}) \} \\
\max\{ -\lambda, G(\bz^n) \}
\end{pmatrix}
\end{equation*}
with
\begin{equation}\label{eq:newtonmatrix}
 \tilde{H}_n := \begin{pmatrix}
D^2_{zz} \tilde{\EE}(\cdot) + \lambda_n \omega M & & & D^2_{zu} \tilde{\EE}(\cdot) & & & \textup{Id}_{N_h \times N_h} & & & \omega M ( \bz^n - \bz^{k-1} ) \\[1ex]
D^2_{zu} \tilde{\EE}(\cdot) & & & D^2_{uu} \tilde{\EE}(\cdot) & & & 0 & & & 0 \\[1ex]
\diag(\balpha^n) & & & 0 & & & \diag(1+\balpha^n) & & & 0 \\[1ex]
\chi_n (\bz^n - \bz^{k-1})^\top \omega M & & & & & & & & & -1 + \chi_n
\end{pmatrix}
\end{equation}
Note that to shorten the notation we let $ (\cdot) \, \hat{=} \, ( t_{k-1}, \bz^n, S_h(\bz^n))$. Moreover, we have 
\begin{equation*} 
\balpha^n_i := 
 \begin{cases}
  -1, & - q_i^n - (z_i^n - z_i^{k-1}) > 0, \\
  0, & - q_i^n - (z_i^n - z_i^{k-1}) \leq 0,  
 \end{cases}
\qquad\text{and}\qquad
 \chi_n := 
\begin{cases}
 1, & G(\bz^n)>-\lambda^n, \\
 0, & G(\bz^n)\leq -\lambda^n.
\end{cases}
\end{equation*}
Eventually, we update $ \bz^{n+1} = \bz^n + \Delta \bz^n $, $\bq^{n+1} = \bq^n + \Delta \bq^n$ and $ \lambda^{n+1} = \lambda^n + \Delta \lambda^n$. With this choice, all matrices $\tilde{H}_n$ appearing in our numerical test have shown to be invertible
and the semi-smooth Newton method performed well with respect to both, robustness and efficiency. In particular, 
no globalization efforts are needed to ensure convergence of the method and, moreover, no line search was necessary in order to guarantee condition \eqref{eq:alg.zupdate_lower} in \nameref{alg:locmin}. 
A rigorous convergence analysis of the method however would go beyond the scope of this paper and is subject to future research.

\end{appendix}

\renewcommand{\refname}{References} 
\bibliography{ref_unbounded}
\bibliographystyle{alphadin}

\end{document}